\documentclass[12pt, letterpaper]{amsart}

\newcommand{\indentalign}{\hspace{0.3in}&\hspace{-0.3in}}
\newcommand{\la}{\langle}
\newcommand{\ra}{\rangle}
\renewcommand{\Re}{\operatorname{Re}}

\newcommand{\defeq}{\stackrel{\rm{def}}{=}}

\newcommand{\op}{\operatorname{op}}
\newcommand{\tr}{\operatorname{Tr}}

\usepackage{amssymb}
\usepackage{amsfonts}
\usepackage{geometry}
\usepackage{graphicx}

\newtheorem{theorem}{Theorem}
\theoremstyle{plain}

\newtheorem{corollary}{Corollary}

\newtheorem{lemma}{Lemma}

\newtheorem{problem}{Problem}
\newtheorem{proposition}{Proposition}
\newtheorem{remark}{Remark}

\numberwithin{equation}{section}
\numberwithin{theorem}{section}  
\numberwithin{proposition}{section}  
\numberwithin{lemma}{section}  
\numberwithin{corollary}{section}  
\input{tcilatex}
\begin{document}
\title[2D NLS from 3D Quantum Many-Body Dynamic]{On the Rigorous Derivation
of the 2D Cubic Nonlinear Schr\"{o}dinger Equation from 3D Quantum Many-Body
Dynamics}
\author{Xuwen Chen}
\address{Department of Mathematics, Brown University, 151 Thayer Street,
Providence, RI 02912}
\email{chenxuwen@math.brown.edu}
\author{Justin Holmer}
\address{Department of Mathematics, Brown University, 151 Thayer Street,
Providence, RI 02912}
\email{holmer@math.brown.edu}
\date{10/15/2012}

\begin{abstract}
We consider the 3D quantum many-body dynamics describing a dilute bose gas with strong confining in one direction.  We study the corresponding BBGKY hierarchy which contains a diverging coefficient as the strength of the confining potential tends to $\infty$.  We find that this diverging coefficient is counterbalanced by the limiting structure of the density matrices and establish the convergence of the BBGKY hierarchy.  Moreover, we prove that the limit is fully described by a 2D cubic NLS and obtain the exact 3D to 2D coupling constant.
\end{abstract}

\maketitle
\tableofcontents

\section{Introduction}

It is widely believed that the cubic nonlinear Schr\"{o}dinger equation (NLS)
\begin{equation*}
i\partial _{t}\phi =L\phi +\left\vert \phi \right\vert ^{2}\phi \text{ in }
\mathbb{R}^{n+1},
\end{equation*}
where $L$ is the Laplacian $-\triangle $ or the Hermite operator $-\triangle
+\omega ^{2}\left\vert x\right\vert ^{2},$ describes the physical phenomenon
of Bose-Einstein condensation (BEC). This belief is one of the main
motivations for studying the cubic NLS. BEC is the phenomenon that particles
of integer spin (bosons) occupy a
macroscopic quantum state. This unusual state of matter was first predicted
theoretically by Einstein for non-interacting particles. The first
experimental observation of BEC in an interacting atomic gas did not occur
until 1995 using laser cooling techniques \cite{Anderson, Davis}. E. A.
Cornell, W. Ketterle, and C. E. Wieman were awarded the 2001 Nobel Prize in
physics for observing BEC. Many similar successful experiments \cite{Cornish, Ketterle, Stamper} were
performed later.

Let $t\in \mathbb{R}$ be the time variable and $\mathbf{r}_{N}=\left(
r_{1},r_{2},...,r_{N}\right) \in \mathbb{R}^{nN}$ be the position vector of $
N$ particles in $\mathbb{R}^{n}$. Then BEC naively means that the $N$-body
wave function $\psi _{N}(t,\mathbf{r}_{N})$ satisfies 
\begin{equation*}
\psi _{N}(t,\mathbf{r}_{N})\sim \dprod\limits_{j=1}^{N}\varphi (t,r_{j})
\end{equation*}
up to a phase factor solely depending on $t$, for some one particle state $
\varphi .$ In other words, every particle is in the same quantum state.
Equivalently, there is the Penrose-Onsager formulation \cite{Penrose} of
BEC: if we define $\gamma _{N}^{(k)}$ to be the $k$-particle marginal
densities associated with $\psi _{N}$ by
\begin{equation}
\gamma _{N\,}^{(k)}(t,\mathbf{r}_{k};\mathbf{r}_{k}^{\prime })=\int \psi
_{N}(t,\mathbf{r}_{k},\mathbf{r}_{N-k})\overline{\psi _{N}}(t,\mathbf{r}
_{k}^{\prime },\mathbf{r}_{N-k})d\mathbf{r}_{N-k},\quad \mathbf{r}_{k},
\mathbf{r}_{k}^{\prime }\in \mathbb{R}^{nk}  \label{E:marginal}
\end{equation}
then, equivalently, BEC means
\begin{equation}
\gamma _{N}^{(k)}(t,\mathbf{r}_{k};\mathbf{r}_{k}^{\prime })\sim
\dprod\limits_{j=1}^{k}\varphi (t,r_{j})\bar{\varphi}(t,r_{j}^{\prime }).
\label{formula:BEC state}
\end{equation}
Gross \cite{Gr1,Gr2} and Pitaevskii \cite{Pitaevskii} proposed that the
many-body effect should be model by a strong on-site interaction and hence
the one-particle state $\varphi $ should be modeled by the a cubic NLS. In a
series of works \cite{Lieb1, LiebAndSeiringer, E-E-S-Y1,
E-S-Y1,E-S-Y2,E-S-Y4, E-S-Y5, E-S-Y3,TChenAndNPSpace-Time, Chen3DDerivation}
, it has been proven rigorously that, under suitable assumptions on the
interaction potential, relation \eqref{formula:BEC state} holds in 3D and the
one-particle state $\varphi $ satisfies the 3D cubic NLS.

It is then natural to believe that the 2D cubic NLS describes the 2D BEC as
well. However, there is no BEC in 2D unless the temperature is absolute zero
(see p. 69 of \cite{Lieb2} and the references within). In other words, 2D
BEC is physically impossible due to the third law of thermodynamics. In a
physically realistic setting, 2D NLS can only arise from a 3D BEC with
strong confining in one direction (which we take to be the $z$-direction).
Such an effective 3D$\ $to 2D phenomenon has been experimentally observed 
\cite{Kettle3Dto2DExperiment, FrenchExperiment, Philips, NatureExperiment,
Another2DExperiment}. (See \cite{ReviewFor2DExperiment} for a review.) It is
then natural to consider the derivation of the 2D NLS from a 3D $N$-body
quantum dynamic. Combining \cite{Abdallah1, Abdallah2, Chen3DDerivation}
suggests a route of getting the 2D NLS from 3D. First, a special case of
Theorem 2 in \cite{Chen3DDerivation} establishes the 3D cubic NLS 
\begin{equation}
i\partial _{t}\varphi =-\triangle _{x}\varphi +\left( -\partial
_{z}^{2}+\omega ^{2}z^{2}\right) \varphi +\left\vert \varphi \right\vert
^{2}\varphi ,\text{ }\left( x,z\right) \in \mathbb{R}^{2+1}
\label{eqn:3D Cubic NLS}
\end{equation}%
from the 3D $N$-body quantum dynamic as a $N\rightarrow \infty $ limit. Then
the result in \cite{Abdallah1, Abdallah2} shows that the 2D cubic NLS arises
from equation \eqref{eqn:3D Cubic NLS} as a $\omega \rightarrow \infty $
limit. This path corresponds to the iterated limit ($\lim_{\omega
\rightarrow \infty }\lim_{N\rightarrow \infty }$) of the $N$-body dynamic,
thus the 2D cubic NLS coming from such a path approximates the 3D $N$-body
dynamic when $\omega $ is large and $N$ is infinity. In experiments, it is
fully possible to have $N$ and $\omega $ comparable to each other. In fact, $%
N$ is about $10^{4}$ and $\omega $ is about $10^{3}$ in \cite%
{Kettle3Dto2DExperiment, FrenchExperiment, NatureExperiment,
Another2DExperiment}. In this paper, we derive rigorously the 2D cubic NLS
as the double limit ($\lim_{N,\omega \rightarrow \infty }$) of a 3D quantum $%
N$-body dynamic directly, without passing through any 3D cubic NLS. 
It is elementary mathematical analysis that $\lim_{\omega \rightarrow \infty
}\lim_{N\rightarrow \infty }$ and $\lim_{N,\omega \rightarrow \infty }$ are
topologically different and one does not imply each other. Let us adopt the
notation 
\begin{equation*}
r_{i}=(x_{i},z_{i})\in \mathbb{R}^{2+1}
\end{equation*}%
and investigate the procedure of laboratory experiments of BEC according to 
\cite{Kettle3Dto2DExperiment, FrenchExperiment, Philips, NatureExperiment,
Another2DExperiment}.

\noindent \textbf{Step A}. Confine a large number of bosons inside a trap with strong
confining in the $z$-direction. Cool it down so that the many-body system
reaches its ground state. It is expected that this ground state is a BEC
state / factorized state. To formulate the problem mathematically, we use
the quadratic potential $\left\vert \cdot \right\vert ^{2}$ to represent the
trap and 
\begin{equation*}
V_{a}\left( r\right) =\frac{1}{a^{3\beta }}V\left( \frac{r}{a^{\beta }}
\right) \text{, }\beta >0
\end{equation*}
to represent the interaction potential. We use the quadratic potential to
represent the trap because this simplified yet reasonably general model is
expected to capture the salient features of the actual trap: on the one hand
the quadratic potential varies slowly, on the other hand it tends to $\infty 
$ as $\left\vert x\right\vert \rightarrow \infty $. In the physics
literature, Lieb, Seiringer and Yngvason remarked in \cite{Lieb1} that the
confining potential is typically $\sim \left\vert x\right\vert ^{2}$ in the
available experiments. The review \cite{ReviewFor2DExperiment} on \cite{Kettle3Dto2DExperiment, FrenchExperiment, Philips, NatureExperiment, Another2DExperiment} also mentioned that the trap is harmonic. We use $
V_{a}\left( r\right) $ to represent the interaction potential to match the
Gross-Pitaevskii description \cite{Gr1,Gr2,Pitaevskii} that the many-body
effect should be modeled by an on-site self interaction because $V_{a}$ is
an approximation of the identity as $a\rightarrow 0$. This step then corresponds
to the following mathematical problem:

\begin{problem}
\label{Problem:Schnee-Yngvason}Show that, for large $N$ and large $\omega \gg \omega _{0}$, the ground state of the $N$-body Hamiltonian
\begin{equation}
\sum_{j=1}^{N}\left( -\triangle _{r_{j}}+\omega _{0}^{2}\left\vert
x_{j}\right\vert ^{2}+\omega ^{2}z_{j}^{2}\right) +
\sum_{1\leqslant i<j\leqslant N}\frac{1}{a^{3\beta -1}}V\left( \frac{
r_{i}-r_{j}}{a^{\beta }}\right)   \label{E:general-Hamiltonian}
\end{equation}
is a factorized state under proper assumptions on $a$ and $V$.
\end{problem}

\noindent \textbf{Step B}. Switch the trap in order to enable measurement or direct
observation. It is assumed that such a shift of the confining potential is
instant and does not destroy the BEC obtained from Step A. To be more
precise about the word ``switch'': in \cite{Philips, Another2DExperiment}, the
trap in the $x$-spatial directions are tuned very loose to generate a 2D
Bose gas. For mathematical convenience, we can assume $\omega _{0}$ becomes $
0$. The system is then time dependent. Therefore, the factorized structure
obtained in Step A must be preserved in time for the observation of BEC.
Mathematically, this step stands for the following problem.

\begin{problem}
\label{Problem:ours}Take the BEC state obtained in Step A as initial datum,
show that, for large $N$ and $\omega ,$ the solution to the $N-$body Schr\"{o}dinger equation
\begin{equation}
i\partial _{t}\psi _{N,\omega }=\sum_{j=1}^{N}\left( -\frac{1}{2}\triangle
_{r_{j}}+\frac{\omega ^{2}}{2}z_{j}^{2}\right) \psi _{N,\omega }+
\sum_{1\leqslant i<j\leqslant N}\frac{1}{a^{3\beta-1 }}V\left( \frac{
r_{i}-r_{j}}{a^{\beta }}\right) \psi _{N,\omega }
\label{equation:N-Body Schrodinger with anisotropic trap}
\end{equation}
is a BEC state / factorized state under the same assumptions of the
interaction potential $V$ in Problem \ref{Problem:Schnee-Yngvason}.
\end{problem}

We first remark that neither of the problems listed above admits a
factorized state solution. It is also unrealistic to solve the equations in
Problems \ref{Problem:Schnee-Yngvason} and \ref{Problem:ours} for large $N$.
Moreover, both problems are linear so that it is not clear how the 2D cubic
NLS arises from either problem. Therefore, in order to justify the statement
that the 2D cubic NLS depicts the 3D to 2D BEC, we have to show
mathematically that, in an appropriate sense, for some 3D one particle state 
$\varphi $ fully described by the 2D cubic NLS
\begin{equation*}
\gamma _{N,\omega }^{(k)}(t,\mathbf{r}_{k};\mathbf{r}_{k}^{\prime })\sim
\dprod\limits_{j=1}^{k}\varphi (t,r_{j})\bar{\varphi}(t,r_{j}^{\prime })
\text{ as }N,\omega \rightarrow \infty
\end{equation*}
where $\gamma _{N,\omega }^{(k)}$ are the $k$-marginal densities associated
with $\psi _{N,\omega }$.

For Problem \ref{Problem:Schnee-Yngvason} (Step A), a satisfying answer has
been found by Schnee and Yngvason. Let $\func{scat}(W)$ denote the 3D
scattering length of the potential $W$. By \cite[Lemma A.1]{E-S-Y2}, for $
0<\beta \leq 1$ and $a\ll 1$, we have 
\begin{equation*}
\func{scat}\left( a\cdot \frac{1}{a^{3\beta }}V\left( \frac{r}{a^{\beta }}
\right) \right) \sim \left\{ \begin{aligned} &a \int_{\mathbb{R}^3} V &&
\text{if } 0\leq \beta <1 \\ &a \func{scat}(V) && \text{if }\beta =1 \end{aligned}
\right. 
\end{equation*}
Consider $\phi _{\omega _{0},Ng}$, the minimizer to the 2D NLS energy
functional 
\begin{equation}
E_{\omega _{0},Ng}=\int_{\mathbb{R}^{2}}\left( |\nabla \phi (x)|^{2}+\omega
_{0}^{2}\left\vert x\right\vert ^{2}|\phi (x)|^{2}+4\pi Ng|\phi
(x)|^{4}\right) \,dx  \label{E:GP-Hamiltonian}
\end{equation}
subject to the constraint $\Vert \phi \Vert _{L^{2}(\mathbb{R}^{2})}=1$. The
existence of this nonlinear ground state stems from the presence of the
confining potential $\omega _{0}^{2}\left\vert x\right\vert ^{2}$; otherwise
the nonlinear term is defocusing (as it is called in the NLS literature).

Given parameters $\omega _{0},\omega ,N,a$, Schnee-Yngvason \cite
{SchneeYngvason} define $g=g(\omega _{0},\omega ,N,a)$ and $\bar{\rho}=\bar{
\rho}(\omega _{0},\omega ,N,a)$ by the two simultaneous equations (see
(1.15) and (1.18) in \cite{SchneeYngvason}) 
\begin{equation*}
g\overset{\mathrm{def}}{=}\left\vert -\log (\frac{\bar{\rho}}{\omega })+
\frac{1}{\sqrt{\omega }a\int_{\mathbb{R}}h_{1}^{4}}\right\vert
^{-1}\,,\qquad \bar{\rho}=N\int |\phi _{\omega _{0},Ng}|^{4}.
\end{equation*}
They argue that this definition for $g$ makes the 2D NLS Hamiltonian 
\eqref{E:GP-Hamiltonian} relevant to the analysis of the limiting behavior
of the ground state of \eqref{E:general-Hamiltonian} describing a dilute
interacting Bose gas in a 3D trap that is strongly confining in the $z$-direction. (See also \cite{JunYin} for the case with rotation)

The Gross-Pitaevskii limit means $Ng\sim 1$. We have liberty to fix the
value of $\omega _{0}$ by scaling, so we take $\omega _{0}=1$. Then the
minimizer $\phi _{\omega _{0},Ng}$ is fixed and hence $\bar{\rho}\sim N$.

In this paper, we consider Problem \ref{Problem:ours} (Step B) and offer a
rigorous derivation of the 2D cubic NLS from the 3D quantum many-body
dynamic. For the scaling of the interaction potential, we consider the case
(called Region I in \cite{SchneeYngvason}) in which the term $(\sqrt{\omega }
a)^{-1}$ dominates in the definition of $g$. Then 
\begin{equation*}
1\sim Ng\sim Na\sqrt{\omega }\iff a\sim \frac{1}{N\sqrt{\omega }}
\end{equation*}
This then implies that 
\begin{equation*}
\frac{1}{\sqrt{\omega }a}\sim N\gg \log \frac{N}{\omega }\sim \log \frac{
\bar{\rho}}{\omega }
\end{equation*}
so that our assumption that the term $(\sqrt{\omega }a)^{-1}$ dominates in
the definition of $g$ is self-consistent.

We will take for mathematical convenience $a=(N\sqrt{\omega })^{-1}$ for
Problem \ref{Problem:ours} (Step B). The Hamiltonian 
\eqref{E:general-Hamiltonian} then becomes 
\begin{equation}
H_{N,\omega }=\sum_{j=1}^{N}\left( -\triangle _{r_{j}}+\omega
^{2}z_{j}^{2}\right) +\frac{1}{N\sqrt{\omega }}\sum_{1\leqslant i<j\leqslant
N}\left( N\sqrt{\omega }\right) ^{3\beta }V\left( \left( N\sqrt{\omega }
\right) ^{\beta }\left( r_{i}-r_{j}\right) \right) 
\label{Hamiltonian:H_N,W,nonscaled}
\end{equation}

Let $h(z)=\pi ^{-1}e^{-z^{2}/2}$ so that $h$ is the normalized ground state
eigenfunction of $-\partial _{z}^{2}+z^{2}$, i.e. it solves $(-1-\partial
_{z}^{2}+z^{2})h=0$. Then the normalized ground state eigenfunction $
h_{\omega }(z)$ of $-\partial _{z}^{2}+\omega ^{2}z^{2}$ is given by $
h_{\omega }(z)=\omega ^{1/4}h(\omega ^{1/2}z)$, i.e. it solves $(-\omega
-\partial _{z}^{2}+\omega ^{2}z^{2})h_{\omega }=0$. In particular, $h_{1}=h$.

We consider initial data that is asymptotically (as $N\rightarrow \infty
,\omega \rightarrow \infty $) factorized in the $x$-direction and in the
ground state in the $z$-direction; in particular we could take 
\begin{equation*}
\psi _{N,\omega }(0,\mathbf{r}_{N})=\prod_{j=1}^{N}\phi _{0}(x_{j})h_{\omega
}(z_{j})\,,\qquad \Vert \phi _{0}\Vert _{L^{2}(\mathbb{R}^{2})}=1.
\end{equation*}
Let 
\begin{equation}
\psi _{N,\omega }(t,\cdot )=e^{itH_{N,\omega }}\psi _{N,\omega }(0,\cdot )
\label{E:evolution}
\end{equation}
denote the evolution of this initial data according to the Hamiltonian 
\eqref{Hamiltonian:H_N,W,nonscaled}. We prove that in a certain sense, as $
N\rightarrow \infty ,\omega \rightarrow \infty $, 
\begin{equation}
\psi _{N,\omega }(t,\mathbf{r}_{N})\sim \prod_{j=1}^{N}\phi
(t,x_{j})h_{\omega }(z_{j})  \label{E:informal-conv}
\end{equation}
where $\phi (t)$ solves a 2D cubic NLS with initial data $\phi _{0}(x)$. To make
this statement more precise, we introduce the rescaled solution 
\begin{equation}
\tilde{\psi}_{N,\omega }(t,\mathbf{r}_{N})\overset{\mathrm{def}}{=}\frac{1}{
\omega ^{N/4}}\psi _{N,\omega }(t,\mathbf{x}_{N},\frac{\mathbf{z}_{N}}{\sqrt{
\omega }})  \label{E:rescaled}
\end{equation}
and the rescaled Hamiltonian 
\begin{equation}
\tilde{H}_{N,\omega }=\sum_{j=1}^{N}(-\Delta _{x_{j}}+\omega (-\partial
_{z_{j}}^{2}+z_{j}^{2}))+\frac{1}{N}\sum_{1\leq i<j\leq N}V_{N,\omega
}(r_{i}-r_{j})  \label{E:rescaled-Hamiltonian}
\end{equation}
where 
\begin{equation}
V_{N,\omega }(r)=N^{3\beta }\left( \sqrt{\omega }\right) ^{3\beta -1}V\left(
\left( N\sqrt{\omega }\right) ^{\beta }x,\frac{\left( N\sqrt{\omega }\right)
^{\beta }}{\sqrt{\omega }}z\right) ,  \label{E:V}
\end{equation}
Then 
\begin{equation*}
(\tilde{H}_{N,\omega }\tilde{\psi}_{N,\omega })(t,\mathbf{x}_{N},\mathbf{z}
_{N})=\frac{1}{\omega ^{N/4}}(H_{N,\omega }\psi _{N,\omega })(t,\mathbf{x}
_{N},\frac{\mathbf{z}_{N}}{\sqrt{\omega }})
\end{equation*}
and hence when $\psi _{N,\omega }(t)$ is given by \eqref{E:evolution} and $
\tilde{\psi}_{N,\omega }$ is defined by \eqref{E:rescaled}, we have 
\begin{equation*}
\tilde{\psi}_{N,\omega }(t,\mathbf{r}_{N})=e^{it\tilde{H}_{N,\omega }}\tilde{
\psi}(0,\mathbf{r}_{N})
\end{equation*}
The informal statement of convergence given by \eqref{E:informal-conv}
becomes the informal statement 
\begin{equation}
\tilde{\psi}(t,\mathbf{r}_{N})\sim \prod_{j=1}^{N}\phi (t,x_{j})h(z_{j})
\label{E:informal-conv-rescaled}
\end{equation}
where $\phi (t)$ solves 2D NLS with initial data $\phi _{0}(x)$. In fact,
the convergence we prove is stated in terms of the associated density
operators with kernels 
\begin{equation}
\tilde{\gamma}_{N,\omega }(t,\mathbf{r}_{N},\mathbf{r}_{N}^{\prime })=\tilde{
\psi}(t,\mathbf{r}_{N})\overline{\tilde{\psi}(t,\mathbf{r}_{N}^{\prime })}
\label{E:densities}
\end{equation}
The version of \eqref{E:informal-conv-rescaled} that we prove is the
convergence 
\begin{equation*}
\tilde{\gamma}_{N,\omega }^{(k)}(t,\mathbf{r}_{k},\mathbf{r}_{k}^{\prime})\rightarrow \prod_{j=1}^{k}\phi (x_{j})h(z_{j})\overline{\phi (x_{j}^{\prime})} \overline{h(z_{j}^{\prime })}
\end{equation*}
in trace class, for each $k\geq 0$.

We define 
\begin{equation}
\label{E:vofbeta}
v(\beta) = \max\left( \frac{1-\beta}{2\beta}, \; \frac{\frac54\beta-\frac1{12}}{1-\frac52\beta}, \; \frac{\frac12\beta + \frac56}{1-\beta}, \; \frac{\beta+\frac13}{1-2\beta}\right)
\end{equation}
(see Fig. \ref{F:vofbeta})

Our main theorem is the following:

\begin{theorem}[main theorem]
\label{Theorem:3D->2D BEC (Nonsmooth)} Assume the pair interaction $V$ is a
nonnegative Schwartz class function. Let $\{\tilde{\gamma}_{N,\omega
}^{(k)}(t,\mathbf{r}_{k};\mathbf{r}_{k}^{\prime })\,\}$ be the family of
marginal densities associated with the 3D rescaled Hamiltonian evolution $
\tilde{\psi}_{N,\omega }(t)=e^{it\tilde{H}_{N,\omega }}\tilde{\psi}
_{N,\omega }(0)$ for some $\beta \in \left( 0,2/5\right) $, (see 
\eqref{E:marginal}, \eqref{E:rescaled-Hamiltonian}, \eqref{E:densities}).
Suppose the initial datum $\tilde{\psi}_{N,\omega }(0)$ satisfies the
following:

\textnormal{(a)} $\tilde \psi _{N,\omega }(0)$ is normalized, that is, $\| \tilde \psi
_{N,\omega }(0)\|_{L^{2}}=1$,

\textnormal{(b)} $\tilde{\psi}_{N,\omega }(0)$ is asymptotically factorized in the sense
that 
\begin{equation*}
\lim_{N,\omega \rightarrow \infty }\limfunc{Tr}\left\vert \tilde{\gamma}
_{N,\omega }^{(1)}(0,x_{1},z_{1};x_{1}^{\prime },z_{1}^{\prime })-\phi
_{0}(x_{1})\overline{\phi _{0}}(x_{1}^{\prime })h(z_{1})h(z_{1}^{\prime
})\right\vert =0,
\end{equation*}
for some one particle state $\phi _{0}\in H^{1}\left( \mathbb{R}^{2}\right)
, $

\textnormal{(c)} Away from the $z$-directional ground state energy, $\tilde{\psi}
_{N,\omega }(0)$ has finite energy per particle: 
\begin{equation*}
\sup_{\omega ,N}\frac{1}{N}\langle \tilde{\psi}_{N,\omega }(0),(\tilde{H}
_{N,\omega }-N\omega )\tilde{\psi}_{N,\omega }(0)\rangle \leqslant C,
\end{equation*}
Then $\forall k\geqslant 1,t\geqslant 0,$ and $\varepsilon >0$, we have the
convergence in trace norm (propagation of chaos) that 
\begin{equation*}
\lim_{\substack{ N,\omega \rightarrow \infty  \\ N\geqslant \omega ^{v(\beta
)+\varepsilon }}}\limfunc{Tr}\left\vert \tilde{\gamma}_{N,\omega }^{(k)}(t,
\mathbf{x}_{k},\mathbf{z}_{k};\mathbf{x}_{k}^{\prime },\mathbf{z}
_{k}^{\prime })-\dprod\limits_{j=1}^{k}\phi (t,x_{j})\overline{\phi }
(t,x_{j}^{\prime })h_{1}(z_{j})h_{1}(z_{j}^{\prime })\right\vert =0,
\end{equation*}
where $v(\beta )$ is given by \eqref{E:vofbeta} and $\phi (t,x)$ solves the
2D cubic NLS with coupling constant $b_{0}\left( \int \left\vert
h_{1}(z)\right\vert ^{4}dz\right) $ that is 
\begin{equation}
i\partial _{t}\phi =-\triangle _{x}\phi +b_{0}\left( \int \left\vert
h_{1}(z)\right\vert ^{4}dz\right) \left\vert \phi \right\vert ^{2}\phi \quad 
\text{ in }\mathbb{R}^{2+1}  \label{equation:2D Cubic NLS}
\end{equation}
with initial condition $\phi \left( 0,x\right) =\phi _{0}(x)$ and $
b_{0}=\int V\left( r\right) dr$.
\end{theorem}

Theorem \ref{Theorem:3D->2D BEC (Nonsmooth)} is equivalent to the following
theorem.

\begin{theorem}[main theorem]
\label{Theorem:3D->2D BEC} Assume the pair interaction $V$ is a nonnegative
Schwartz class function. Let $\{\tilde{\gamma}_{N,\omega }^{(k)}(t,\mathbf{r}
_{k};\mathbf{r}_{k}^{\prime })\,\}$ be the family of marginal densities
associated with the 3D rescaled Hamiltonian evolution $\tilde{\psi}
_{N,\omega }(t)=e^{it\tilde{H}_{N,\omega }}\tilde{\psi}_{N,\omega }(0)$ for
some $\beta \in \left( 0,2/5\right) $, (see \eqref{E:marginal}, 
\eqref{E:rescaled-Hamiltonian}, \eqref{E:densities}). Suppose the initial
datum $\tilde{\psi}_{N,\omega }(0)$ is normalized, asymptotically factorized
and satisfies the energy condition that

\textnormal{($\text{c}'$)} there is a $C>0$ such that 
\begin{equation}
\langle \tilde{\psi}_{N,\omega }(0),(\tilde{H}_{N,\omega }-N\omega )^{k}
\tilde{\psi}_{N,\omega }(0)\rangle \leqslant C^{k}N^{k}\text{, }\forall
k\geqslant 1,  \label{Condition:EnergyBoundOnInitialData}
\end{equation}
Then $\forall k\geqslant 1,t\geqslant 0,$ and $\varepsilon >0$, we have the
convergence in trace norm (propagation of chaos) that 
\begin{equation*}
\lim_{\substack{ N,\omega \rightarrow \infty  \\ N\geqslant \omega ^{v(\beta
)+\varepsilon }}}\limfunc{Tr}\left\vert \tilde{\gamma}_{N,\omega }^{(k)}(t,
\mathbf{x}_{k},\mathbf{z}_{k};\mathbf{x}_{k}^{\prime },\mathbf{z}
_{k}^{\prime })-\dprod\limits_{j=1}^{k}\phi (t,x_{j})\overline{\phi }
(t,x_{j}^{\prime })h_{1}(z_{j})h_{1}(z_{j}^{\prime })\right\vert =0,
\end{equation*}
where $v(\beta )$ is given by \eqref{E:vofbeta} and $\phi (t,x)$ solves the
2D cubic NLS \eqref{equation:2D Cubic NLS}.
\end{theorem}

\begin{figure}
\includegraphics[scale=0.7]{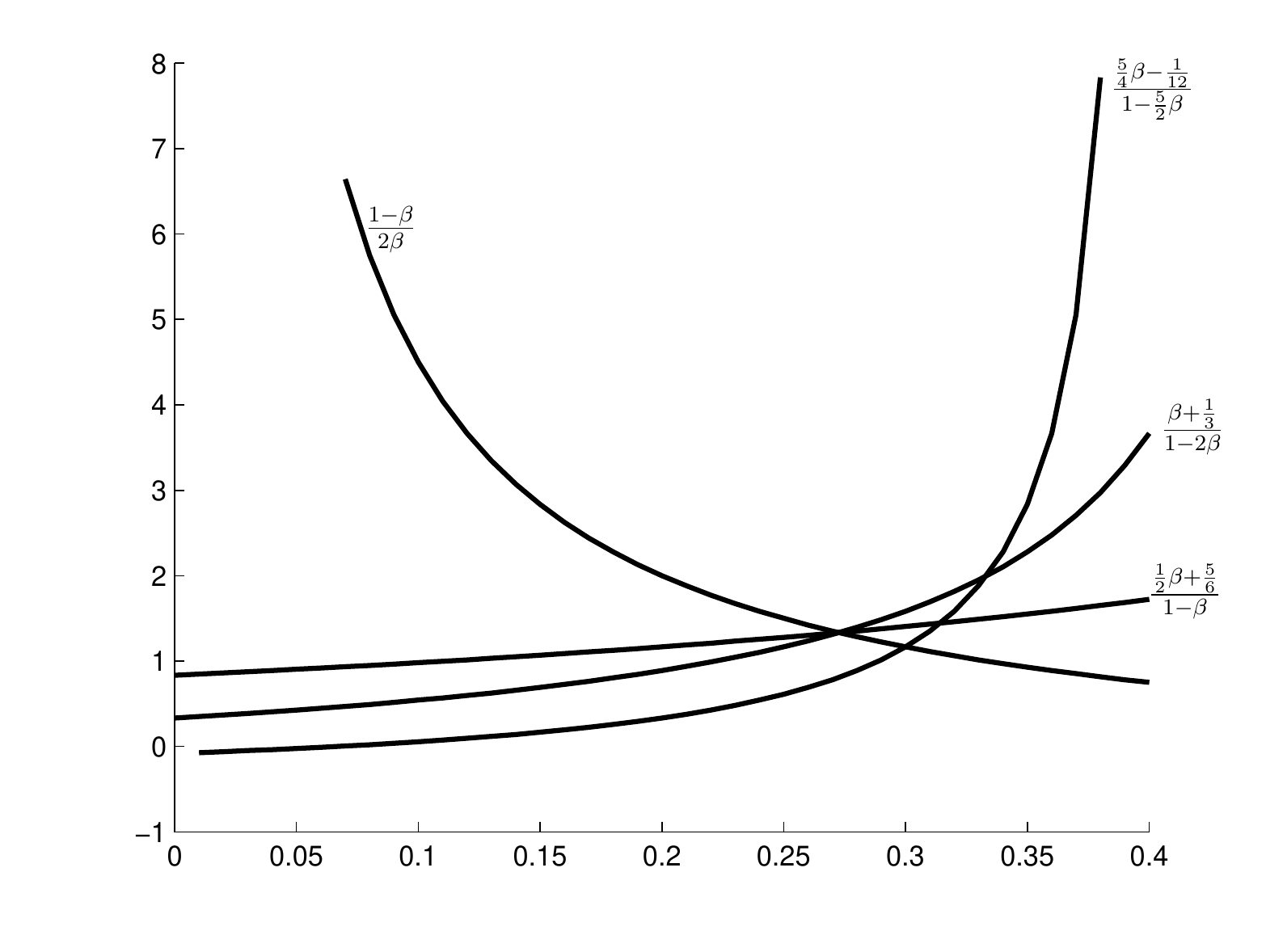}
\caption{ \label{F:vofbeta} A graph of the various rational functions of $\beta$ appearing in \eqref{E:vofbeta}.  In Theorems \ref{Theorem:3D->2D BEC (Nonsmooth)}, \ref{Theorem:3D->2D BEC}, the limit $(N,\omega) \to \infty$ is taken with $N\geq \omega^{v(\beta)+\epsilon}$.  As shown here, there are values of $\beta$ for which $v(\beta) \sim 1$, which allows $N\sim \omega$, as in the experimental paper \cite{Kettle3Dto2DExperiment, FrenchExperiment, NatureExperiment,Another2DExperiment}.  We conjecture that Theorems \ref{Theorem:3D->2D BEC (Nonsmooth)}, \ref{Theorem:3D->2D BEC} hold with \eqref{E:vofbeta} replaced by the weaker constraint $v(\beta)=\frac{1-\beta}{2\beta}$ for all $0< \beta < 1$.}
\end{figure}

We remark that assumptions (a), (b), and (c) in Theorem \ref{Theorem:3D->2D BEC (Nonsmooth)} are
reasonable assumptions on the initial datum coming from Step A. In
fact, if we assume further that $\phi _{0}$ minimizes the 2D
Gross-Pitaevskii functional \eqref{E:GP-Hamiltonian}, then (a), (b) and (c)
are the conclusion of \cite[Theorem 1.1, 1.3]{SchneeYngvason}. The limit in
Theorem \ref{Theorem:3D->2D BEC (Nonsmooth)}, which is taken as $N,\omega \rightarrow
\infty $ within the subregion $N\geqslant \omega ^{v(\beta )+\varepsilon }$
is optimal in the sense that if $N\leqslant \omega ^{\frac{1}{2\beta }-\frac{
1}{2}}$, then the limit of $V_{N,\omega }$ defined by \eqref{E:V} is not a
delta function.

The equivalence of Theorems \ref{Theorem:3D->2D BEC (Nonsmooth)} and \ref{Theorem:3D->2D BEC} for asymptotically factorized initial data is well-known. In the main part of this paper, we prove Theorem \ref{Theorem:3D->2D BEC} in full detail. For completeness, we
discuss briefly how to deduce Theorem \ref{Theorem:3D->2D BEC (Nonsmooth)} from Theorem \ref{Theorem:3D->2D BEC} in Appendix \ref{A:equivalence}.

The main tool used to prove Theorem \ref{Theorem:3D->2D BEC} is the analysis
of the BBGKY hierarchy of $\left\{ \tilde{\gamma}_{N,\omega }^{(k)}\right\}
_{k=1}^{N}$ as $N,\omega \rightarrow \infty .$ With our definition, the
sequence of the marginal densities $\left\{ \tilde{\gamma}_{N,\omega
}^{(k)}\right\} _{k=1}^{N}$ associated with $\tilde{\psi}_{N,\omega }$
satisfies the BBGKY hierarchy

\begin{equation}
\label{hierarchy:BBGKY hierarchy for scaled marginal densities}
\begin{aligned}
i\partial _{t}\tilde{\gamma}_{N,\omega }^{(k)} =&\sum_{j=1}^{k}\left[
-\triangle _{x_{j}},\tilde{\gamma}_{N,\omega }^{(k)}\right]
+\sum_{j=1}^{k}\omega \left[ -\partial _{z_{j}}^{2}+z_{j}^{2},\tilde{\gamma}
_{N,\omega }^{(k)}\right] +\frac{1}{N}\sum_{i<j}^{k}\left[ V_{N,\omega
}\left( r_{i}-r_{j}\right) ,\tilde{\gamma}_{N,\omega }^{(k)}\right]\\
&+\frac{N-k}{N}\limfunc{Tr}\nolimits_{r_{k+1}}\sum_{j=1}^{k}\left[
V_{N,\omega }\left( r_{j}-r_{k+1}\right) ,\tilde{\gamma}_{N,\omega }^{(k+1)}
\right]
\end{aligned}
\end{equation}

In the classical setting, deriving mean-field type equations by studying the
limit of the BBGKY hierarchy was proposed by Kac and demonstrated by
Landford's work \cite{Lanford} on the Boltzmann equation. In the quantum
setting, the usage of the BBGKY hierarchy was suggested by Spohn \cite{Spohn}
and has been proven to be successful by Elgart, Erd\"{o}s, Schlein, and Yau
in their fundamental papers \cite{E-E-S-Y1, E-S-Y1,E-S-Y2,E-S-Y4, E-S-Y5,
E-S-Y3} which rigorously derives the 3D cubic NLS from a 3D quantum many-body
dynamic without a trap. The Elgart-Erd\"{o}s-Schlein-Yau program consists of
two principal parts: in one part, they consider the sequence of the marginal
densities $\left\{ \gamma _{N}^{(k)}\right\} $ associated with the
Hamiltonian evolution $e^{itH_{N}}\psi _{N}(0)$ where
\begin{equation*}
H_{N}=\sum_{j=1}^{N}-\triangle _{r_{j}}+\frac{1}{N}\sum_{1\leqslant
i<j\leqslant N}N^{3\beta }V(N^{\beta }\left( r_{i}-r_{j}\right) )
\end{equation*}
and prove that an appropriate limit of as $N\rightarrow \infty $ solves the
3D Gross-Pitaevskii hierarchy 
\begin{equation}
i\partial _{t}\gamma ^{(k)}+\sum_{j=1}^{k}\left[ \triangle _{r_{k}},\gamma
^{(k)}\right] =b_{0}\sum_{j=1}^{k}\limfunc{Tr}\nolimits_{r_{k+1}}
[\delta(r_j-r_{k+1}),\gamma ^{(k+1)}] ,\text{ for all }k \geq 1 \,.
\label{equation:Gross-Pitaevskii hiearchy without a trap}
\end{equation}
In another part, they show that hierarchy \eqref{equation:Gross-Pitaevskii
hiearchy without a trap} has a unique solution which is therefore a
completely factorized state. However, 
the uniqueness theory for hierarchy \eqref{equation:Gross-Pitaevskii hiearchy
without a trap} is surprisingly delicate due to the fact that it is a
system of infinitely many coupled equations over an unbounded number of
variables. In \cite{KlainermanAndMachedon}, by imposing a space-time bound
on the limit of $\left\{ \gamma _{N}^{(k)}\right\} $, Klainerman and
Machedon gave another proof of the uniqueness in \cite{E-S-Y2} through a
collapsing estimate originating from the ordinary multilinear Strichartz
estimates in their null form paper \cite{KlainermanMachedonNullForm} and a
board game argument inspired by the Feynman graph argument in \cite{E-S-Y2}.

Later, the method in Klainerman and Machedon \cite{KlainermanAndMachedon}
was taken up by Kirkpatrick, Schlein, and Staffilani \cite{Kirpatrick}, who
derived the 2D cubic NLS from the 2D quantum many-body dynamic; by Chen and
Pavlovi\'{c} \cite{TChenAndNpGP1, TChenAndNP}, who considered the 1D and 2D
3-body interaction problem and the general existence theory of hierarchy $
\eqref{equation:Gross-Pitaevskii hiearchy without a trap}$; and by X.C. \cite
{ChenAnisotropic}, who investigated the trapping problem in 2D and 3D. In 
\cite{TCNPNT, TCNPNT1}, Chen, Pavlovi\'{c} and Tzirakis worked out the
virial and Morawetz identities for hierarchy \eqref{equation:Gross-Pitaevskii
hiearchy without a trap}. In 2011, for the 3D case without traps, Chen and
Pavlovi\'{c} \cite{TChenAndNPSpace-Time} proved that, for $\beta \in (0,1/4)$
, the limit of $\left\{ \gamma _{N}^{(k)}\right\} $ actually satisfies the
space-time bound assumed by Klainerman and Machedon \cite
{KlainermanAndMachedon} as $N\rightarrow \infty $. This has been a
well-known open problem in the field. In 2012, X.C. \cite{Chen3DDerivation}
extended and simplified their method to study the 3D trapping problem for $
\beta \in (0,2/7].$

The $\beta =0$ case has been studied by many authors as well \cite{E-Y1,LChen,KnowlesAndPickl,MichelangeliSchlein,RodnianskiAndSchlein}. 

Away from the usage of the BBGKY hierarchy, there has been work by X.C.,
Grillakis, Machedon and Margetis \cite{GMM1,GMM2,Chen2ndOrder,GM1} using the
second order correction which can deal with $e^{itH_{N}}\psi _{N}$ directly.

To our knowledge, this is the first direct rigorous treatment of the 3D to 2D dynamic problem.  We now compare our theorem with the known work which derives $n$D cubic NLS from the $n$D quantum many-body dynamic.  It is easy to tell that Theorem \ref{Theorem:3D->2D BEC} deals with a different limit than the known  work \cite{AGT, E-E-S-Y1,
E-S-Y1,E-S-Y2,E-S-Y4, E-S-Y5,
E-S-Y3,Kirpatrick,TChenAndNP,ChenAnisotropic,TChenAndNPSpace-Time,
Chen3DDerivation} which derives $n$D NLS from $n$D dynamics. On the one
hand, Theorem \ref{Theorem:3D->2D BEC} deals with a 3D to 2D effect. Such a
phenomenon is described by the limit equation \eqref{equation:2D Cubic NLS}
and the coupling constant $\int \left\vert h_{1}(z)\right\vert ^{4}dz.$ The
limit in Theorem \ref{Theorem:3D->2D BEC} is with the scaling
\begin{equation*}
\lim_{\substack{ N,\omega \rightarrow \infty  \\ N\geqslant \omega ^{v(\beta
)+\varepsilon }}}N\sqrt{\omega }\limfunc{scat}\left( \frac{V_{N,\omega }}{N}
\right) =\text{constant,}
\end{equation*}
instead of the scaling
\begin{equation*}
\lim_{N\rightarrow \infty }N\limfunc{scat}(N^{n\beta -1}V(N^{\beta }\cdot ))=
\text{constant,}
\end{equation*}
in the known $n$D to $n$D work.

The main idea of the proof of Theorem \ref{Theorem:3D->2D BEC} is to
investigate the limit of hierarchy \eqref{hierarchy:BBGKY hierarchy for scaled
marginal densities} which at a glance is similar to the $n$D to $n$D work.
However, in contrast with the $n$D to $n$D case, even the formal limit of
hierarchy \eqref{hierarchy:BBGKY hierarchy for scaled marginal densities} is
not known.

Heuristically, according to the uncertainty principle, in 3D, as the $z$-component of the particles' position becomes more and more determined to be $0$, the $z$-component of the momentum and thus the energy must blow up. Hence the energy of the system is dominated by its $z$-directional part which is in fact infinity as $N,\omega \rightarrow \infty $. This renders the energy and thus the analysis of the $x-$component intractable.

Technically, it is not clear whether the term
\begin{equation*}
\omega \left[ -\partial _{z_{j}}^{2}+z_{j}^{2},\tilde{\gamma}_{N,\omega
}^{(k)}\right] 
\end{equation*}
tends to a limit as $N,\omega \rightarrow \infty $. Since $\tilde{\gamma}
_{N,\omega }^{(k)}\ $is not a factorized state for $t>0$, one cannot expect
the commutator to be zero. Thus we formally have an $\infty -\infty $ in
hierarchy \eqref{hierarchy:BBGKY hierarchy for scaled marginal densities} as $
N,\omega \rightarrow \infty .$ This is the main difficulty we need to
circumvent in the proof of Theorem \ref{Theorem:3D->2D BEC}.

\subsection{Acknowledgements}
J.H. was supported in part by NSF grant DMS-0901582 and a Sloan Research Fellowship (BR-4919). X.C. would like to express his thanks to M. Grillakis, M. Machedon, D. Margetis, W. Strauss, and N. Tzirakis for discussions related to this work, to T. Chen and N. Pavlovi\'{c} for raising the 2D to 1D question during the X.C.'s seminar talk in Austin, to K. Kirkpatrick for encouraging X.C. to work on this problem during X.C.'s visit to Urbana.  We thank Christof Sparber for pointing out references \cite{Abdallah1, Abdallah2}.

\section{Outline of the proof of Theorem \ref{Theorem:3D->2D BEC}}

We begin by setting down some notation that will be used in the remainder of the paper.  We will always assume $\omega \geq 1$.  Note that, as an operator, we have the positivity:
$$-1 - \partial_{z_j}^2 + z_j^2 \geq 0$$
Define 
\begin{equation}
\label{E:tilde-S-def}
\tilde S_j \defeq (1-\Delta_{x_j} + \omega( -1 - \partial_{z_j}^2 + z_j^2))^{1/2}
\end{equation}
We have $\tilde S_j^2 (\phi(x_j) h(z_j)) = (1-\Delta_{x_j})\phi(x_j) \, h(z_j)$ and thus the diverging $\omega$ parameter has no consequence when the operator is applied to a tensor product function $\phi(x_j)h(z_j)$ for which the $z_j$-component rests in the ground state. 

Let $P_0$ denote the orthogonal projection onto the ground state of $-\partial_z^2 + z^2$ and $P_1$ denote the orthogonal projection onto all higher energy modes, so $I=P_0+P_1$, where $I:L^2(\mathbb{R}^3)\to L^2(\mathbb{R}^3)$.  Let $P_0^j$ and $P_1^j$ be the corresponding operators acting on $L^2(\mathbb{R}^{3N})$ in the $z_j$ component, $1\leq j \leq N$.  Then 
\begin{equation}
\label{E:cpct1}
I = \prod_{j=1}^k (P_0^j+ P_1^j) \,, \quad \text{where} \quad I: L^2(\mathbb{R}^{3N}) \to L^2(\mathbb{R}^{3N})
\end{equation}  
For a $k$-tuple $\mathbf{\alpha} = (\alpha_1, \ldots, \alpha_k)$ with $\alpha_j \in \{0,1\}$, let $P_{\mathbf{\alpha}} = P_{\alpha_1}^1 \cdots P_{\alpha_k}^k$.   Adopt the notation
$$|\mathbf{\alpha}| = \alpha_1 + \cdots + \alpha_k$$
This leads to the coercivity (operator lower bounds) given in Lemma \ref{L:coercivity}.

We next introduce an appropriate topology on the density matrices as was previously done in \cite{E-E-S-Y1, E-Y1, E-S-Y1,E-S-Y2,E-S-Y4, E-S-Y5, E-S-Y3,Kirpatrick,TChenAndNP,ChenAnisotropic,Chen3DDerivation}.  Denote the spaces of compact operators and trace class operators on $L^{2}\left( \mathbb{R}^{3k}\right) $ as $\mathcal{K}_{k}$ and $\mathcal{L}_{k}^{1}$, respectively.  Then $\left( \mathcal{K}_{k}\right) ^{\prime }=\mathcal{L}_{k}^{1}$. By the fact that $\mathcal{K}_{k}$ is separable, we select a dense countable subset $\{ J_{i}^{(k)}\} _{i\geqslant 1}\subset \mathcal{K}_{k}$ in the unit ball of $\mathcal{K}_{k}$ (so $\Vert J_{i}^{(k)}\Vert _{\op}\leqslant 1$ where $\left\Vert \cdot \right\Vert_{\op}$ is the operator norm).  For $\gamma ^{(k)},\tilde{\gamma}^{(k)}\in \mathcal{L}_{k}^{1}$, we then define a metric $d_{k}$ on $\mathcal{L}_{k}^{1}$ by 
\begin{equation*}
d_{k}(\gamma ^{(k)},\tilde{\gamma}^{(k)})=\sum_{i=1}^{\infty
}2^{-i}\left\vert \limfunc{Tr}J_{i}^{(k)}\left( \gamma ^{(k)}-\tilde{\gamma}
^{(k)}\right) \right\vert .
\end{equation*}
A uniformly bounded sequence $\tilde{\gamma}_{N,\omega }^{(k)}\in 
\mathcal{L}_{k}^{1}$ converges to $\tilde{\gamma}^{(k)}\in \mathcal{L}
_{k}^{1}$ with respect to the weak* topology if and only if 
\begin{equation*}
\lim_{N,\omega \rightarrow \infty }d_{k}(\tilde{\gamma}_{N,\omega }^{(k)},
\tilde{\gamma}^{(k)})=0.
\end{equation*}
For fixed $T>0$, let $C\left( \left[ 0,T\right] ,\mathcal{L}_{k}^{1}\right) $
be the space of functions of $t\in \left[ 0,T\right] $ with values in $
\mathcal{L}_{k}^{1}$ which are continuous with respect to the metric $d_{k}.$
On $C\left( \left[ 0,T\right] ,\mathcal{L}_{k}^{1}\right) ,$ we define the
metric
\begin{equation*}
\hat{d}_{k}(\gamma ^{(k)}\left( \cdot \right) ,\tilde{\gamma}^{(k)}\left(
\cdot \right) )=\sup_{t\in \left[ 0,T\right] }d_{k}(\gamma ^{(k)}\left(
t\right) ,\tilde{\gamma}^{(k)}\left( t\right) ),
\end{equation*}
and denote by $\tau _{prod}$ the topology on the space $\oplus _{k\geqslant
1}C\left( \left[ 0,T\right] ,\mathcal{L}_{k}^{1}\right) $ given by the
product of topologies generated by the metrics $\hat{d}_{k}$ on $C\left( 
\left[ 0,T\right] ,\mathcal{L}_{k}^{1}\right) .$

With the above topology on the space of marginal densities, we now outline the proof of Theorem \ref{Theorem:3D->2D BEC}. We divide the proof into five steps.

\medskip

\noindent \textbf{Step I} (Energy estimate).  We transform, through Theorem \ref{Theorem:Energy Estimate},
the energy condition \eqref{Condition:EnergyBoundOnInitialData} into an
``easier to use'' $H^{1}$ type energy bound in which the interaction $V$ is not involved. Since the quantity on the left-hand side of energy condition \eqref{Condition:EnergyBoundOnInitialData} is conserved by the evolution, we
deduce the \emph{a priori} bounds on the scaled marginal densities
$$
\sup_{t}\limfunc{Tr}\dprod\limits_{j=1}^{k}\left( 1-\triangle
_{x_{j}}+\omega \left( -1-\partial _{z_{j}}^{2}+z_{j}^{2}\right) \right) 
\tilde{\gamma}_{N,\omega }^{(k)} \leqslant C^{k} 
$$
$$
\sup_{t}\limfunc{Tr}\dprod\limits_{j=1}^{k}\left( 1-\triangle
_{r_{j}}\right) \tilde{\gamma}_{N,\omega }^{(k)} \leqslant C^{k}
$$
$$
\sup_t \tr P_{\mathbf{\alpha}} \tilde \gamma_{N,\omega}^{(k)} P_{\mathbf{\beta}} \leq C^k \omega^{-\frac12 |\mathbf{\alpha}| - \frac12 \mathbf{|\beta}|}
$$
via Corollary \ref{Corollary:Energy Bound for Marginal Densities}. We remark
that, in contrast to the $n$D to $n$D work, the quantity 
\begin{equation*}
\limfunc{Tr}\left( 1-\triangle _{r_{1}}\right) \tilde{\gamma}_{N,\omega
}^{(1)}
\end{equation*}
is not the one particle kinetic energy of the system; the one particle
kinetic energy  of the system is $\limfunc{Tr}\left(
1-\triangle _{x_{1}}-\omega \partial _{z_{1}}^{2}\right) \tilde{\gamma}
_{N,\omega }^{(1)}$ and grows like $\omega$.

\medskip

\noindent\textbf{Step II} (Compactness of BBGKY). We fix $T>0$ and work in the time-interval $t\in \lbrack
0,T].$ In Theorem \ref{Theorem:Compactness of the scaled marginal density},
we establish the compactness of the sequence $\Gamma _{N,\omega }(t)=\left\{ 
\tilde{\gamma}_{N,\omega }^{(k)}\right\} _{k=1}^{N}\in \oplus _{k\geqslant
1}C\left( \left[ 0,T\right] ,\mathcal{L}_{k}^{1}\right) $ with respect to
the product topology $\tau _{prod}$ even though there is an $\infty -\infty $ in
hierarchy \eqref{hierarchy:BBGKY hierarchy for scaled marginal densities}.
Moreover, in Corollary \ref{Corollary:LimitMustBeAProduct}, we prove that, to
be compatible with the energy bound obtained in Step I, every limit point $
\Gamma (t)=\left\{ \tilde{\gamma}^{(k)}\right\} _{k=1}^{N}$ must take the
form
\begin{equation*}
\tilde{\gamma}^{(k)}\left( t,\left( \mathbf{x}_{k},\mathbf{z}_{k}\right)
;\left( \mathbf{x}_{k}^{\prime },\mathbf{z}_{k}^{\prime }\right) \right) =
\tilde{\gamma}_{x}^{(k)}(t,\mathbf{x}_{k};\mathbf{x}_{k}^{\prime
})\dprod\limits_{j=1}^{k}h_{1}\left( z_{j}\right) h_{1}\left( z_{j}^{\prime
}\right) ,
\end{equation*}
where $\tilde{\gamma}_{x}^{(k)}=\limfunc{Tr}_{z}\tilde{\gamma}^{(k)}$ is the 
$x$-component of $\tilde{\gamma}^{(k)}.$

\medskip

\noindent\textbf{Step III} (Limit points of BBGKY satisfy GP). In Theorem \ref{Theorem:Convergence to the Coupled
Gross-Pitaevskii}, we prove that if $\Gamma (t)=\left\{ \tilde{\gamma}
^{(k)}\right\} _{k=1}^{\infty }$ is a $N\geqslant \omega ^{v(\beta)+\varepsilon }$ limit point of $\Gamma _{N,\omega }(t)=\left\{ 
\tilde{\gamma}_{N,\omega }^{(k)}\right\} _{k=1}^{N}$ with respect to the
product topology $\tau _{prod}$, then $\left\{ \tilde{\gamma}_{x}^{(k)}=
\limfunc{Tr}_{z}\tilde{\gamma}^{(k)}\right\} _{k=1}^{\infty }$ is a solution
to the coupled Gross-Pitaevskii (GP) hierarchy subject to initial data $\tilde{
\gamma}_{x}^{(k)}\left( 0\right) =\left\vert \phi _{0}\right\rangle
\left\langle \phi _{0}\right\vert ^{\otimes k}$ with coupling constant $
b_{0}=$ $\int V\left( r\right) dr$, which written in differential form, is
\begin{equation*}
i\partial _{t}\tilde{\gamma}_{x}^{(k)}=\sum_{j=1}^{k}\left[ -\triangle
_{x_{j}},\tilde{\gamma}_{x}^{(k)}\right] +b_{0}\sum_{j=1}^{k}\limfunc{Tr}
\nolimits_{x_{k+1}}\limfunc{Tr}\nolimits_{z}\left[ \delta \left(
r_{j}-r_{k+1}\right) ,\tilde{\gamma}^{(k+1)}\right] .
\end{equation*}
Together with Corollary \ref{Corollary:LimitMustBeAProduct}, we then deduce
that $\left\{ \tilde{\gamma}_{x}^{(k)}=\limfunc{Tr}_{z}\tilde{\gamma}
^{(k)}\right\} _{k=1}^{\infty }$ is a solution to the well-known 2D
GP hierarchy subject to initial data $\tilde{\gamma}
_{x}^{(k)}\left( 0\right) =\left\vert \phi _{0}\right\rangle \left\langle
\phi _{0}\right\vert ^{\otimes k}$ with coupling constant $b_{0}\left( \int
\left\vert h_{1}\left( z\right) \right\vert ^{4}dz\right) $, which, written
in differential form, is
\begin{equation}
i\partial _{t}\tilde{\gamma}_{x}^{(k)}=\sum_{j=1}^{k}\left[ -\triangle
_{x_{j}},\tilde{\gamma}_{x}^{(k)}\right] +b_{0}\left( \int \left\vert
h_{1}\left( z\right) \right\vert ^{4}dz\right) \sum_{j=1}^{k}\limfunc{Tr}
\nolimits_{x_{k+1}}\left[ \delta \left( x_{j}-x_{k+1}\right) ,\tilde{\gamma}
_{x}^{(k+1)}\right] .  \label{hierarchy:2D GP hierarchy in differential form}
\end{equation}

\medskip

\noindent\textbf{Step IV} (GP has a unique solution). When $\tilde{\gamma}_{x}^{(k)}\left( 0\right) =\left\vert
\phi _{0}\right\rangle \left\langle \phi _{0}\right\vert ^{\otimes k},$ we
know one solution to the 2D Gross-Pitaevskii hierarchy \eqref{hierarchy:2D GP
hierarchy in differential form}, namely $\left\vert \phi \right\rangle \left\langle \phi \right\vert ^{\otimes k}$, where $\phi $ solves equation \eqref{equation:2D Cubic NLS}. Since we have the
\emph{a priori} bound
\begin{equation*}
\sup_{t}\limfunc{Tr}\dprod\limits_{j=1}^{k}\left( 1-\triangle
_{x_{j}}\right) \tilde{\gamma}_{x}^{(k)}\leqslant C^{k},
\end{equation*}
the uniqueness theorem $($Theorem \ref{Theorem:CombiningChenAndKirpatrick}$)$
then gives that $\tilde{\gamma}_{x}^{(k)}=\left\vert \phi \right\rangle \left\langle \phi\right\vert ^{\otimes k}$. Thus the compact sequence $\Gamma _{N,\omega }(t)=\left\{ \tilde{\gamma}_{N,\omega }^{(k)}\right\} _{k=1}^{N}$ has only one $N\geqslant \omega^{v(\beta)+\varepsilon }$ limit point, namely
\begin{equation*}
\tilde{\gamma}^{(k)}=\dprod\limits_{j=1}^{k}\phi (t,x_{j})\overline{\phi }
(t,x_{j}^{\prime })h_{1}\left( z_{j}\right) h_{1}(z_{j}^{\prime }) \,.
\end{equation*}
By the definition of the topology, we know, as trace class operators
\begin{equation*}
\tilde{\gamma}_{N,\omega }^{(k)}\rightarrow \dprod\limits_{j=1}^{k}\phi
(t,x_{j})\overline{\phi }(t,x_{j}^{\prime })h_{1}\left( z_{j}\right)
h_{1}(z_{j}^{\prime })\text{ weak*.}
\end{equation*}

\begin{remark}
This is in fact the very first time that the Klainerman-Machedon theory
applies to a 3D many-body system with $\beta \geqslant 1/3$. The previous
best is $\beta \in \left( 0,2/7\right] $ in \cite{Chen3DDerivation} after
the $\beta \in \left( 0,1/4\right) $ work \cite{TChenAndNPSpace-Time}. Of
course, we are not actually using any 3D Gross-Pitaevskii hierarchies here.
\end{remark}

\noindent\textbf{Step V} (Weak convergence upgraded to strong).  We use the argument in the bottom of p. 296 of \cite{E-S-Y3} to conclude that the weak* convergence obtained in Step IV is in fact strong. We include this argument for completeness. We test the sequence obtained in Step IV against the compact observable
\begin{equation*}
J^{(k)}=\dprod\limits_{j=1}^{k}\phi (t,x_{j})\overline{\phi }
(t,x_{j}^{\prime })h_{1}\left( z_{j}\right) h_{1}(z_{j}^{\prime }),
\end{equation*}
and notice the fact that $\left( \tilde{\gamma}_{N,\omega }^{(k)}\right)
^{2}\leqslant \tilde{\gamma}_{N,\omega }^{(k)}$ since the initial data is
normalized, we see that as Hilbert-Schmidt operators
\begin{equation*}
\tilde{\gamma}_{N,\omega }^{(k)}\rightarrow \dprod\limits_{j=1}^{k}\phi
(t,x_{j})\overline{\phi }(t,x_{j}^{\prime })h_{1}\left( z_{j}\right)
h_{1}(z_{j}^{\prime })\text{ strongly.}
\end{equation*}
Since $\limfunc{Tr}\tilde{\gamma}_{N,\omega }^{(k)}=\limfunc{Tr}\tilde{\gamma
}^{(k)},$ we deduce the strong convergence

\begin{equation*}
\lim_{\substack{ N,\omega \rightarrow \infty  \\ N\geqslant \omega ^{v(\beta
)+\varepsilon }}}\limfunc{Tr}\left\vert 
\tilde \gamma _{N,\omega }^{(k)}(t,\mathbf{x}_{k},\mathbf{z}_{k};\mathbf{x}_{k}^{\prime },\mathbf{z}_{k}^{\prime})-\dprod\limits_{j=1}^{k}\phi (t,x_{j})\overline{\phi }
(t,x_{j}^{\prime })h_{1}\left( z_{j}\right) h_{1}(z_{j}^{\prime
})\right\vert =0,
\end{equation*}
via the Gr\"{u}mm's convergence theorem \cite[Theorem 2.19]{Simon}

\section{Energy estimate\label{Section:EnergyEstimate}}

We find it more convenient to prove the energy estimate for $\psi_{N,\omega}$ and then convert it by scaling to an estimate for $\tilde\psi_{N,\omega}$ (see \eqref{E:rescaled}).  Note that, as an operator, we have the positivity:
$$-\omega - \partial_{z_j}^2 + \omega^2 z_j^2 \geq 0$$
Define 
$$S_j \defeq (1-\Delta_{x_j} - \omega - \partial_{z_j}^2 + \omega^2 z_j^2)^{1/2} = (1-\omega - \Delta_{r_j} + \omega^2z_j^2)^{1/2}$$


\begin{theorem}
\label{Theorem:Energy Estimate}
Let the Hamiltonian be defined as in \eqref{Hamiltonian:H_N,W,nonscaled} with $\beta \in \left( 0,2/5\right)$.  Then for all $\varepsilon >0$, there exists a constant $C>0$, and for all $\omega,k\geqslant 0$, there exists $N_{0}(k,\omega)$ such that
\begin{equation}
\label{E:energy-nonscaled}
\left\langle \psi_{N,\omega} ,\left( N+H_{N,\omega }-N\omega \right) ^{k}\psi_{N,\omega} \right\rangle \geqslant C^{k}N^{k}  \left\| \prod_{j=1}^k S_j \psi_{N,\omega} \right\|_{L^2(\mathbb{R}^{3N})}^2
\end{equation}
for all $N\geqslant \omega^{v(\beta)+\epsilon}$, and all $\psi \in L_{s}^{2}\left( \mathbb{R}^{3N}\right)\cap \mathcal{D}(H_{N,\omega}^k)$.  
\end{theorem}

\begin{proof}
We adapt the proof of \cite[Prop. 3.1]{E-E-S-Y1} to accommodate the operator $-\omega - \partial_{z_j}^2 + \omega^2z_j^2$ in place of $-\partial_{z_j}^2$.
The case $k=0$ is trivial and the case $k=1$ follows from the positivity of $V$ and symmetry of $\psi$. We proceed by induction.  Suppose that the result holds for $k=n$, and we will prove it for $k=n+2$.  By the induction hypothesis,
\begin{equation}
\label{E:en1}
\begin{aligned}
\indentalign \la \psi, (N-N\omega+H_{N,\omega})^{n+2}\psi \ra \\
&\geq C^nN^n \la \psi, (N-N\omega+H_{N,\omega}) \prod_{j=1}^n S_j^2 (N-N\omega+H_{N,\omega}) \psi \ra
\end{aligned}
\end{equation}
For convenience, let
$$\tilde V(r) = (N\sqrt \omega)^{3\beta-1} V( (N\sqrt \omega)^\beta r)$$
Expand
$$N-N\omega + H_{N,\omega} = \sum_{\ell=n+1}^N S_\ell^2 + \left(\sum_{\ell=1}^n S_\ell^2 + H_{N,\omega}^I\right)$$
and substitute in both occurrences of the operator $N-N\omega + H_{N,\omega}$ in the right side of \eqref{E:en1} to obtain four terms.  We ignore the last (positive) one of these terms to obtain
\begin{equation}
\label{E:en2}
\la \psi, (N-N\omega+H_{N,\omega})^{n+2}\psi \ra \geq C^nN^n(\text{I}+\text{II}+\text{III})
\end{equation}
We have
$$\text{I} =  \sum_{\ell_1,\ell_2 = n+1}^N \la \psi, S_{\ell_1}^2S_{\ell_2}^2 \prod_{j=1}^n S_j^2  \psi \ra$$
In this double sum, there are $(N-n)(N-n-1)$ terms where $\ell_1\neq \ell_2$ that are all the same by symmetry, and there are $(N-n)$ terms where $\ell_1=\ell_2$ that are all the same by symmetry.  We have
\begin{equation}
\label{E:en11}
\text{I} = (N-n)(N-n-1)\la \psi , \prod_{j=1}^{n+2} S_j^2 \psi \ra +(N-n) \la \psi, S_1^2 \prod_{j=1}^{n+1} S_j^2 \psi \ra
\end{equation}
the first of which will ultimately fulfill the induction claim.
In \eqref{E:en2}, we also have
$$\text{II}+\text{III} = 
\begin{aligned}[t]
&2\sum_{\ell_1=n+1}^N \sum_{\ell_2=1}^n \la \psi , S_{\ell_1}^2 \prod_{j=1}^n S_j^2 S_{\ell_2}^2 \psi \ra + \sum_{\ell=n+1}^N \la \psi, S_\ell^2 \prod_{j=1}^n S_j^2 H_{N,\omega}^I \psi \ra \\
&+ \sum_{\ell=n+1}^N \la \psi, H_{N,\omega}^I \prod_{j=1}^n S_j^2 S_\ell^2\psi \ra
\end{aligned}
$$
Exploiting symmetry this becomes
\begin{equation}
\label{E:en10}
\text{II}+\text{III} = 2(N-n)n \la \psi, S_1^2 \prod_{j=1}^{n+1} S_j^2 \psi \ra +2(N-n) \Re\la \psi, \prod_{j=1}^{n+1} S_j^2 H_{N,\omega}^I \psi \ra
\end{equation}
In the first term, we have applied the permutation that swaps $\ell_1$ and $n+1$ and $\ell_2$ and $1$.  In the second and third terms, we have applied the permutation $\sigma$ that swaps $\ell$ and $n+1$.  Strictly speaking, this permutation maps $H_{N,\omega}^I$ to $H_{N,\omega,\sigma}^I$ where
$$H_{N,\omega,\sigma}^I \defeq \frac{1}{N\omega^{1/2}} \sum_{1\leq i < j \leq N} (N\omega^{1/2})^{3\beta} V((\pm 1)(N\omega^{1/2})^\beta (r_i-r_j ))$$
where $\pm 1$ is chosen according to the affect of the permutation on the pair $(i,j)$.  The distinction between $H_{N,\omega}^I$ and $H_{N,\omega,\sigma}^I$ is inconsequential for the remainder of the analysis (and in fact $H_{N,\omega}^I=H_{N,\omega,\sigma}^I$ if $V$ is even), so we have ignored it in \eqref{E:en10}.
The first of the terms in \eqref{E:en10} is positive -- it is the second term that requires attention; in particular, we have to manage commutators.  

Assuming $N \geq 2n+2$, we substitute \eqref{E:en11}, \eqref{E:en10} into \eqref{E:en2} to obtain
\begin{equation}
\label{E:en12}
\begin{aligned}
\la \psi, (N-N\omega +H_{N,\omega})^{n+2}\psi \ra \geq \tfrac14 C^nN^{n+2} \la \psi, \prod_{j=1}^{n+2} S_j^2 \psi \ra + C^nN^{n+1} \la \psi, S_1^2 \prod_{j=1}^{n+1} S_j^2 \psi \ra &\\
+ 2C^nN^n (N-n) \Re \la \psi, \prod_{j=1}^{n+1}S_j^2 H_{N,\omega}^I \psi \ra =: D+E+F &
\end{aligned}
\end{equation}
The first two terms, $D$ and $E$, in \eqref{E:en12} are positive.  The third term $F$ will be decomposed into components, some of which are positive and others that can be bounded in terms of the first two terms appearing in \eqref{E:en12}.  In the expression for $H_{N,\omega}^I$, there are
\begin{itemize}
\item $\frac12 (n+1)n$ terms of the form $\tilde V(r_i-r_j)$ for $1\leq i< j \leq n+1$.
\item $(n+1)(N-n-1)$ terms of the form $\tilde V(r_i-r_j)$ for $1\leq i \leq n+1$ and $n+2 \leq j \leq N$.
\item $\frac12 (N-n-1)(N-n-2)$ terms of the form $\tilde V(r_i-r_j)$ for $n+2 \leq i<j \leq N$.
\end{itemize}
For convenience, let
$$V_{ij} \defeq (N\omega^{1/2})^{3\beta -1} V( (N\omega^{1/2})^\beta(r_i-r_j))$$
Using symmetry, we obtain
\begin{align*}
F &= 
\begin{aligned}[t]
&2C^nN^n (N-n) (n+1)n  \Re \la \psi, \prod_{j=1}^{n+1} S_j^2 V_{12} \psi \ra  \\
&+ 2C^nN^n (N-n) (n+1)(N-n-1) \Re \la \psi, \prod_{j=1}^{n+1} S_j^2 V_{1(n+2)} \psi \ra  \\
&+ C^nN^n (N-n)(N-n-1)(N-n-2) \Re \la \psi, \prod_{j=1}^{n+1} S_j^2 V_{(n+2)(n+3)} \psi \ra 
\end{aligned}\\
&=: F_1+F_2+F_3
\end{align*}
The last term $F_3$ is positive since each $S_j$ for $1\leq j \leq n+1$ commutes with $V_{(n+2)(n+3)}$.  We will show $F_1 \geq -\frac12 E$ and $F_2 \geq -\frac12D$ provided $N\geq N_0(n)$, which together with \eqref{E:en12} will complete the induction argument.  We have
\begin{align*}
F_1 &= 2C^nN^n (N-n) (n+1)n  \Re \la \psi,  \prod_{j=1}^{n+1} S_j^2 V_{12} \psi \ra \\
&=  2C^nN^n (N-n) (n+1)n \Re \int_{r_3,\ldots, r_N} \underbrace{\la f, S_1^2S_2^2 V_{12} f \ra_{r_1,r_2}}_{=: \tilde F_1} \, dr_3 \cdots dr_N
\end{align*}
where $f = \prod_{j=3}^{n+1} S_j \psi$.  We can regard $r_3, \ldots, r_N$ as frozen in the following computation, so to prove $|F_1| \leq \frac12 E$, it will suffice to show that 
\begin{equation}
\label{E:en13}
 | \tilde F_1 | \leq  \tfrac14 n^{-2} \|S_1^2S_2 f \|_{L_{r_1}^2L_{r_2}^2}^2
\end{equation}
Toward this end, we have
\begin{align*}
|\tilde F_1 |
&=  |\la S_1^2 f, V_{12} S_2^2 f \ra + 2 \la S_1^2 f, \nabla_{r_2} V_{12} \cdot \nabla_{r_2} f \ra + \la S_1^2 f, (\Delta_{r_2} V_{12}) \, f \ra|\\
&\lesssim \|S_1^2 f \|_{L_{r_1}^2 L_{r_2}^6} \|V_{12} \|_{L_{r_1}^\infty L_{r_2}^3} \|S_2^2 f\|_{L_{r_1}^2 L_{r_2}^2} + \|S_1^2 f \|_{L_{r_1}^2 L_{r_2}^6} \|\nabla_{r_2} V_{12} \|_{L_{r_1}^\infty L_{r_2}^{3/2}} \|\nabla_{r_2} f\|_{L_{r_1}^2 L_{r_2}^6} \\
& \qquad + \| S_1^2 f\|_{L_{r_1}^2L_{r_2}^6} \| \Delta_{r_2} V_{12} \|_{L_{r_1}^\infty L_{r_2}^{6/5}} \|f\|_{L_{r_1}^2 L_{r_2}^\infty}
\end{align*}
By evaluation of 
$$\|V_{12} \|_{ L_{r_2}^3} \sim (N\omega^{1/2})^{2\beta-1}\,, \quad \|\nabla_{r_2} V_{12} \|_{L_{r_2}^{3/2}} \sim (N\omega^{1/2})^{2\beta-1}\,, \quad
\| \Delta_{r_2} V_{12} \|_{L_{r_2}^{6/5}} \sim (N\omega^{1/2})^{\frac52\beta-1}
$$
the above estimate reduces to
\begin{align*}
|\tilde F_1 | &\lesssim (N\omega^{1/2})^{2\beta-1} \|S_1^2 f \|_{L_{r_1}^2 L_{r_2}^6} \|S_2^2 f\|_{L_{r_1}^2 L_{r_2}^2} + (N\omega^{1/2})^{2\beta-1} \|S_1^2 f \|_{L_{r_1}^2 L_{r_2}^6} \|\nabla_{r_2} f\|_{L_{r_1}^2 L_{r_2}^6} \\
& \qquad + (N\omega^{1/2})^{\frac52\beta-1} \| S_1^2 f\|_{L_{r_1}^2L_{r_2}^6}  \|f\|_{L_{r_1}^2 L_{r_2}^\infty} 
\end{align*}
Applying Lemma \ref{L:Sobolev-with-loss}, this reduces further to
\begin{align*}
|\tilde F_1 | &\lesssim (N\omega^{1/2})^{2\beta-1} \omega^{1/6} \|S_1^2 S_2 f \|_{L_{r_1}^2 L_{r_2}^2} \|S_2^2 f\|_{L_{r_1}^2 L_{r_2}^2} \\
&\qquad + (N\omega^{1/2})^{2\beta-1} \omega^{1/6} \omega^{2/3} \|S_1^2S_2 f \|_{L_{r_1}^2 L_{r_2}^2} \| S_2^2 f\|_{L_{r_1}^2 L_{r_2}^2} \\
& \qquad + (N\omega^{1/2})^{\frac52\beta-1} \omega^{1/6} \omega^{1/4} \| S_1^2S_2 f\|_{L_{r_1}^2L_{r_2}^2}  \|S_2^2f\|_{L_{r_1}^2 L_{r_2}^2}
\end{align*}
Hence we need $\beta < \frac25$ and conditions \eqref{E:en20}, \eqref{E:en18} below to achieve \eqref{E:en13}.

Let us now establish $F_2 \geq -\frac12 D$. We have
\begin{align*}
F_2 &= 2C^nN^n (N-n) (n+1)(N-n-1) \Re \la \psi, \prod_{j=1}^{n+1} S_j^2 V_{1(n+2)} \psi \ra \\
&= 2C^nN^n (N-n) (n+1)(N-n-1) \int \underbrace{\la f, S_1^2V_{1(n+2)} f \ra_{r_1,r_{n+2}}}_{=:\tilde F_2} \, dr_2 \cdots dr_{n+1}dr_{n+3}\cdots dr_N 
\end{align*}
where $f = \prod_{j=2}^{n+1} S_j \psi$.  Now
\begin{align*}
\tilde F_2 &= \la f, (-\omega - \partial_{z_1}^2 + \omega^2 z_1^2) V_{1(n+2)} f\ra_{r_1,r_{n+2}} \\
&= -\omega \la f, V_{1(n+2)}f \ra_{r_1r_{n+2}} + \la \partial_{z_1}f, (\partial_{z_1}V_{1(n+2)})f \ra_{r_1r_{n+2}} \\
& \qquad + \la \partial_{z_1}f, V_{1(n+2)} \partial_{z_1}f\ra_{r_1r_{n+2}} + \la f, \omega^2 z_1^2 f\ra_{r_1r_{n+2}} \\
&=: \tilde F_{2,1} + \tilde F_{2,2} + \tilde F_{2,3} + \tilde F_{2,4}
\end{align*}
Note that $\tilde F_{2,3}$ and $\tilde F_{2,4}$ are positive and can thus be disregarded.  To prove $F_2 \geq - \frac12 D$, it suffices to prove
\begin{equation}
\label{E:en14}
|\tilde F_{2,1}|+ |\tilde F_{2,2}| \leq \tfrac1{16} n^{-1} \| S_1 S_{n+2} f \|_{L_{r_1}^2 L_{r_{n+2}}^2}^2
\end{equation}
But
$$|\tilde F_{2,1}| \lesssim \omega \| f\|_{L_{r_1}^2L_{r_{n+2}}^6} \|V_{1(n+2)}\|_{L_{r_1}^\infty L_{r_{n+2}}^{3/2}} \|f\|_{L_{r_1}^2 L_{r_{n+2}}^6}$$
By Lemma \ref{L:Sobolev-with-loss} and $\|V_{1(n+2)} \|_{L_{r_1}^\infty L_{r_{n+2}}^{3/2}} \sim (N\omega^{1/2})^{\beta-1}$, we obtain
\begin{equation}
\label{E:en15}
|\tilde F_{2,1}| \lesssim \omega^{4/3} (N\omega^{1/2})^{\beta-1} \| S_{n+2}f\|_{L_{r_1}^2L_{r_{n+2}}^2}^2
\end{equation}
The upper bound in \eqref{E:en14} will be achieved provided \eqref{E:en19} below holds.  Also,
$$
|\tilde F_{2,2}| \lesssim \| \partial_{z_1} f\|_{L_{r_1}^2 L_{r_{n+2}}^6} \|\partial_{z_1} V_{1(n+2)} \|_{L_{r_1}^\infty L_{r_{n+2}}^{3/2}} \|f\|_{L_{r_1}^2 L_{r_{n+2}}^6}
$$
Note that $\|\partial_{z_1} V_{1(n+2)} \|_{L_{r_1}^\infty L_{r_{n+2}}^{3/2}} \sim (N\omega^{1/2})^{2\beta-1}$.  By Lemma \ref{L:Sobolev-with-loss}, 
$$\| \partial_{z_1} f \|_{L_{r_1}^2 L_{r_{n+2}}^6} \lesssim \omega^{1/6} \| S_{n+2} \partial_{z_1} f\|_{L_{r_1}^2L_{r_{n+2}}^2} \lesssim \omega^{2/3} \|S_1 S_{n+2} f\|_{L_{r_1}^2 L_{r_{n+2}}^2}$$
and $\|f \|_{L_{r_1}^2 L_{r_{n+2}}^6} \lesssim \omega^{1/6} \|S_{n+2}f\|_{L_{r_1}^2L_{r_{n+2}}^2}$.  From this, it follows that
\begin{equation}
\label{E:en16}
|\tilde F_{2,2}| \lesssim \omega^{5/6} (N\omega^{1/2})^{2\beta-1} \| S_1S_{n+2} f\|_{L_{r_1}^2 L_{r_{n+2}}^2} \|S_{n+2} f\|_{L_{r_1}^2 L_{r_{n+2}}^2}
\end{equation}
The upper bound in \eqref{E:en14} will be achieved provided \eqref{E:en20} holds.  By \eqref{E:en15}, \eqref{E:en16}, we obtain \eqref{E:en14}, completing the proof.  Let us collect the conditions on $N$ and $\omega$.  We have
\begin{align}
\label{E:en18} &(N\omega^{1/2})^{\frac52\beta-1}\omega^{5/12} \ll n^{-2} && \iff N \gg \omega^\frac{\frac54 \beta - \frac1{12}}{1-\frac52\beta} n^\frac{2}{1-\frac52\beta} \\
\label{E:en19} &(N\omega^{1/2})^{\beta-1} \omega^{4/3} \ll n^{-1} && \iff N \gg \omega^\frac{ \frac12\beta+\frac56 }{1-\beta} n^\frac{1}{1-\beta}\\
\label{E:en20} &(N\omega^{1/2})^{2\beta-1} \omega^{5/6} \ll n^{-1} && \iff N \gg \omega^ \frac{\beta+\frac13}{1-2\beta} n^\frac{1}{1-2\beta}
\end{align}
The requirement that \eqref{E:en18}, \eqref{E:en19}, and \eqref{E:en20} hold is imposed in the definition \eqref{E:vofbeta} of $v(\beta)$.
\end{proof}

Now consider the rescaled operator \eqref{E:tilde-S-def} so that 
$$(S_j \psi)(t,\mathbf{x}_N, \mathbf{z}_N) = \omega^{N/4} (\tilde S_j \tilde \psi)(t, \mathbf{x}_N, \sqrt \omega \mathbf{z}_N)\,.$$

We will convert the conclusions of Theorem \ref{Theorem:Energy Estimate} into statements about $\tilde \psi$, $\tilde S_j$, and $\tilde \gamma_{N,\omega}^{(k)}$ that we will then apply in the remainder of the paper.

\begin{corollary}
\label{Corollary:Energy Bound for Marginal Densities}
Let $\tilde \psi _{N,\omega}( t) =e^{it \tilde H_{N,\omega }}\tilde \psi _{N,\omega }( 0)$ and $\{ \tilde \gamma _{N,\omega }^{(k)}(t)\} $ be the marginal densities associated with it, then for all $\omega \geq 1$ , $k\geq 0$, $N \geq \omega^{v(\beta)+\epsilon}$, we have the uniform-in-time bound
\begin{equation}
\label{E:e-1}
\tr \prod_{j=1}^k \tilde S_j^2 \tilde \gamma_{N,\omega}^{(k)} = \left\| \prod_{j=1}^k \tilde S_j \tilde \psi_{N,\omega}(t) \right\|_{L^2(\mathbb{R}^{3N})}^2 \leq C^k
\end{equation}
Consequently,
\begin{equation}
\label{E:e-2}
\tr \prod_{j=1}^k (1-\Delta_{r_j}) \tilde \gamma_{N,\omega}^{(k)} = \left\| \prod_{j=1}^k (1-\Delta_{r_j})^{1/2} \tilde \psi_{N,\omega}(t) \right\|_{L^2(\mathbb{R}^{3N})}^2 \leq C^k
\end{equation}
and
\begin{equation}
\label{E:e-3}
\| P_{\mathbf{\alpha}} \tilde \psi_{N,\omega} \|_{L^2(\mathbb{R}^{3N})} \leq C^k \omega^{-|\mathbf{\alpha}|/2}\,, \qquad \tr P_{\mathbf{\alpha}} \tilde \gamma_{N,\omega}^{(k)} P_{\mathbf{\beta}} \leq C^k \omega^{-\frac12 |\mathbf{\alpha}| - \frac12 \mathbf{|\beta}|}
\end{equation}
\end{corollary}

\begin{proof}
Substituting \eqref{E:rescaled} into \eqref{E:energy-nonscaled} of Theorem \ref{Theorem:Energy Estimate} and rescaling, we obtain
\begin{equation}
\label{E:en100}
\la \tilde \psi_{N,\omega}, (N-\tilde H_{N,\omega} - N\omega)^k \tilde \psi_{N,\omega} \ra \geq C^kN^k \left\| \prod_{j=1}^k \tilde S_j \tilde \psi_{N,\omega} \right\|_{L^2(\mathbb{R}^{3N})}^2
\end{equation}
Since $N-\tilde H_{N,\omega} - N\omega$ is self-adjoint and $[\tilde H_{N,\omega}, N- \tilde H_{N,\omega} - N\omega] =0$, 
$$\partial_t \la \tilde \psi_{N,\omega}, (N-\tilde H_{N,\omega} - N\omega)^k \tilde \psi_{N,\omega} \ra =0$$
Hence by \eqref{E:en100},
\begin{align*}
C^k N^k \left\| \prod_{j=1}^k \tilde S_j \tilde \psi_{N,\omega}(t) \right\|_{L^2(\mathbb{R}^{3N})}^2 \leq  
&\la \tilde \psi_{N,\omega}(t), (N- \tilde H_{N,\omega}-N\omega)^k \tilde \psi_{N,\omega}(t) \ra\\
&= \la \tilde \psi_{N,\omega}(0), (N-\tilde H_{N,\omega} - N\omega)^k \tilde \psi_{N,\omega}(0) \ra \leq (C')^k N^k
\end{align*}
where the last estimate follows from the hypothesis \eqref{Condition:EnergyBoundOnInitialData} of Theorem \ref{Theorem:3D->2D BEC}. 

The inequality \eqref{E:e-2} follows from \eqref{E:e-1} and \eqref{E:tilde-S-1}.  The inequality on the left of \eqref{E:e-3} follows from \eqref{E:tilde-S-3} and \eqref{E:e-1}.  By Lemma \ref{L:trace-of-tp-kernel}, $\tr P_{\mathbf{\alpha}} \tilde \gamma_{N,\omega}^{(k)} P_{\mathbf{\beta}} = \la P_{\mathbf{\alpha}} \tilde \psi_{N,\omega}, P_{\mathbf{\beta}} \tilde\psi_{N,\omega} \ra$, so the inequality  on the right of \eqref{E:e-3} follows by Cauchy-Schwarz.
\end{proof}

\section{Compactness of the BBGKY sequence \label{Section:Compactness} }

\begin{theorem}
\label{Theorem:Compactness of the scaled marginal density} 
The sequence 
$$\Gamma _{N,\omega }(t)=\left\{ \tilde{\gamma}_{N,\omega }^{(k)}\right\}_{k=1}^{N}\in \bigoplus _{k\geqslant 1}C\left( \left[ 0,T\right] ,\mathcal{L}_{k}^{1}\right)$$ 
which satisfies the $\infty-\infty$ BBGKY hierarchy \eqref{hierarchy:BBGKY hierarchy for scaled marginal densities}, is compact with respect to the product topology $\tau_{prod}$.  For any limit point $\Gamma (t)=\left\{ \tilde{\gamma}^{(k)}\right\} _{k=1}^{N},$ $\tilde{\gamma}^{(k)}$ is a symmetric nonnegative trace class operator with trace bounded by $1$.
\end{theorem}


We establish Theorem \ref{Theorem:Compactness of the scaled marginal density}
at the end of this section. With Theorem \ref{Theorem:Compactness of the
scaled marginal density}, we can start talking about the limit points of $
\Gamma _{N,\omega }(t)=\{ \tilde{\gamma}_{N,\omega }^{(k)}\}
_{k=1}^{N}.$

\begin{corollary}
\label{Corollary:LimitMustBeAProduct}
Let $\Gamma (t)=\{ \tilde{\gamma}^{(k)}\} _{k=1}^{\infty }$ be a limit point of $\Gamma _{N,\omega}(t)=\{ \tilde{\gamma}_{N,\omega }^{(k)}\} _{k=1}^{N}$ with respect to the product topology $\tau _{prod}$, then $\tilde{\gamma}^{(k)}$ satisfies

\begin{equation}
\label{E:e-7}
\limfunc{Tr}\dprod\limits_{j=1}^{k}\left( 1-\triangle _{r_{j}}\right) \tilde{\gamma}^{(k)}\leqslant C^{k}
\end{equation}

\begin{equation}
\label{E:e-8}
\tilde{\gamma}^{(k)}\left( t,\left( \mathbf{x}_{k},\mathbf{z}_{k}\right);\left( \mathbf{x}_{k}^{\prime },\mathbf{z}_{k}^{\prime }\right) \right) = \tilde{\gamma}_{x}^{(k)}(t,\mathbf{x}_{k};\mathbf{x}_{k}^{\prime})\dprod\limits_{j=1}^{k}h_{1}\left( z_{j}\right) h_{1}\left( z_{j}^{\prime}\right)
\end{equation}
\end{corollary}

\begin{proof}
The estimate \eqref{E:e-7} is a direct consequence of \eqref{E:e-2} in Corollary \ref{Corollary:Energy Bound for Marginal Densities} and Theorem \ref{Theorem:Compactness of the scaled marginal density}.   The formula \eqref{E:e-8} is equivalent to the statement that if either $\mathbf{\alpha} \neq 0$ or $\mathbf{\beta} \neq 0$, then $P_\mathbf{\alpha} \tilde \gamma^{(k)} P_{\mathbf{\beta}} =0$.  This is equivalent to the statement that for any $J^{(k)}\in \mathcal{K}_k$, $\tr J^{(k)} P_\mathbf{\alpha} \tilde \gamma^{(k)} P_{\mathbf{\beta}} = 0$.  However,
\begin{equation}
\label{E:e-4}
\tr J^{(k)} P_\mathbf{\alpha} \tilde \gamma^{(k)} P_{\mathbf{\beta}}= \lim_{(N,\omega) \to \infty} \tr J^{(k)} P_\mathbf{\alpha} \tilde \gamma_{N,\omega}^{(k)} P_{\mathbf{\beta}}
\end{equation}
By Lemma \ref{L:trace-of-tp-kernel},
$$\tr J^{(k)} P_\mathbf{\alpha} \tilde \gamma_{N,\omega}^{(k)} P_{\mathbf{\beta}} = \la J^{(k)} P_{\mathbf{\alpha}} \tilde\psi_{N,\omega}, P_{\mathbf{\beta}} \tilde\psi_{N,\omega} \ra_{\mathbf{r}_k} $$
and by Cauchy-Schwarz and \eqref{E:e-3},
$$| \tr J^{(k)} P_\mathbf{\alpha} \tilde \gamma_{N,\omega}^{(k)} P_{\mathbf{\beta}} |  \leq  \| J^{(k)} \|_{\op} \| P_{\mathbf{\alpha}} \tilde \psi_{N,\omega}\|_{L^2(\mathbb{R}^{3N})} \|  P_{\mathbf{\beta}} \tilde\psi_{N,\omega}\|_{L^2(\mathbb{R}^{3N})} \leq C^k \omega^{-\frac12 |\mathbf{\alpha}| - \frac12 | \mathbf{\beta}|}$$
Hence the right side of \eqref{E:e-4} is $0$.
\end{proof}

\begin{proof}[Proof of Theorem \ref{Theorem:Compactness of the scaled marginal density}]

By the standard diagonalization argument, it suffices to show the compactness of $\tilde{\gamma}_{N,\omega }^{(k)}$ for fixed $k$ with respect to the metric $\hat{d}_{k}$.  By the Arzel\`a-Ascoli theorem, this is equivalent to the equicontinuity of $\tilde{\gamma}_{N,\omega }^{(k)}$, and by \cite[Lemma 6.2]{E-S-Y3}, this is equivalent to the statement that for every observable $J^{(k)}$ from a dense subset of $\mathcal{K}( L^2( \mathbb{R}^{3k}))$ and for every $\varepsilon >0$, there exists $\delta(J^{(k)},\varepsilon )$ such that for all $t_{1},t_{2}\in \left[ 0,T\right]$ with $\left\vert t_{1}-t_{2}\right\vert \leqslant \delta$, we have
\begin{equation}
\label{E:cpct5}
\sup_{N,\omega }\left\vert \tr J^{(k)}\tilde{\gamma}_{N,\omega }^{(k)}(t_{1}) - \tr J^{(k)} \tilde{\gamma}_{N,\omega}^{(k)}(t_{2}) \right\vert \leqslant \varepsilon\, .
\end{equation}

We assume that our compact operators $J^{(k)}$ have been cutoff as in Lemma \ref{L:compact-operator-truncation}.  Assume $t_1\leq t_2$.    Inserting the decomposition \eqref{E:cpct1} on the left and right side of $\gamma_{N,\omega}^{(k)}$, we obtain
$$\tilde \gamma_{N,\omega}^{(k)} = \sum_{\mathbf{\alpha}, \mathbf{\beta}} P_{\mathbf{\alpha}} \tilde \gamma_{N,\omega}^{(k)} P_{\mathbf{\beta}}$$
where the sum is taken over all $k$-tuples $\mathbf{\alpha}$ and $\mathbf{\beta}$ of the type described above. 

 To establish \eqref{E:cpct5} it suffices to establish, for each $\mathbf{\alpha}$ and $\mathbf{\beta}$
\begin{equation}
\label{E:cpct4}
\sup_{N,\omega }\left\vert \tr J^{(k)}P_{\mathbf{\alpha}} \tilde{\gamma}_{N,\omega }^{(k)}P_{\mathbf{\beta}}(t_{1}) - \tr J^{(k)} P_{\mathbf{\alpha}}\tilde{\gamma}_{N,\omega}^{(k)}P_{\mathbf{\beta}}(t_{2})  \right\vert \leqslant \varepsilon\, .
\end{equation}
Below, we establish the estimate
\begin{equation}
\label{E:cpct2}
\begin{aligned}
\indentalign
| \tr J^{(k)} P_{\mathbf{\alpha}} \tilde{\gamma}_{N,\omega }^{(k)} P_{\mathbf{\beta}}(t_2) - \tr J^{(k)} P_{\mathbf{\alpha}} \tilde{\gamma}_{N,\omega }^{(k)} P_{\mathbf{\beta}}(t_1)| \\ 
&\lesssim |t_2-t_1|\begin{cases} 
1 & \text{if both } \mathbf{\alpha}=0 \text{ and }\mathbf{\beta}=0 \\
\max(1,\omega^{1-\frac12 |\mathbf{\alpha}| - \frac12 |\mathbf{\beta}|}) & \text{otherwise}
\end{cases}
\end{aligned}
\end{equation}
Estimate \eqref{E:cpct2} suffices to prove \eqref{E:cpct4} except when $|\mathbf{\alpha}|=0$ and $|\mathbf{\beta}|=1$ or vice versa, in which case it yields the upper bound $\omega^{1/2} |t_2-t_1|$ with the adverse factor $\omega^{1/2}$.  On the other hand, we can also prove the (comparatively simpler) bound 
\begin{equation}
\label{E:cpct3}
| \tr J^{(k)} P_{\mathbf{\alpha}} \tilde{\gamma}_{N,\omega}^{(k)} P_{\mathbf{\beta}}(t_2) - \tr J^{(k)} P_{\mathbf{\alpha}} \tilde{\gamma}_{N,\omega}^{(k)} P_{\mathbf{\beta}}(t_1)|  \lesssim \omega^{-\frac12 |\mathbf{\alpha}| - \frac12 |\mathbf{\beta}|}
\end{equation}
that provides no gain as $t_2\to t_1$, but a better power of $\omega$.  By averaging \eqref{E:cpct2} and \eqref{E:cpct3} in the case $|\mathbf{\alpha}|=0$ and $|\mathbf{\beta}|=1$ (or vice versa), we obtain
$$| \tr J^{(k)} P_{\mathbf{\alpha}} \tilde{\gamma}_{N,\omega}^{(k)} P_{\mathbf{\beta}}(t_2) - \tr J^{(k)} P_{\mathbf{\alpha}} \tilde{\gamma}_{N,\omega}^{(k)} P_{\mathbf{\beta}}(t_1)|  \lesssim |t_2-t_1|^{1/2}$$
which suffices to establish \eqref{E:cpct4}.  

Thus, it remains to prove both \eqref{E:cpct2} and \eqref{E:cpct3}, and we begin with \eqref{E:cpct2}.  Hierarchy \eqref{hierarchy:BBGKY hierarchy for scaled marginal densities} yields
\begin{equation}
\label{E:cpct6}
i\partial _{t} \, P_{\mathbf{\alpha}} \tilde{\gamma}_{N,\omega }^{(k)} P_{\mathbf{\beta}} =
\begin{aligned}[t]
&\sum_{j=1}^{k}\left[-\triangle _{x_{j}}, \; P_{\mathbf{\alpha}} \tilde{\gamma}_{N,\omega }^{(k)} P_{\mathbf{\beta}} \right]
+\sum_{j=1}^{k}\omega \left[ -\partial _{z_{j}}^{2}+z_{j}^{2},  \; P_{\mathbf{\alpha}}\tilde{\gamma}_{N,\omega }^{(k)}  P_{\mathbf{\beta}} \right] \\
&+\frac{1}{N}\sum_{i<j}^{k} P_{\mathbf{\alpha}} \left[ V_{N,\omega}\left( r_{i}-r_{j}\right) , \; \tilde{\gamma}_{N,\omega }^{(k)}\right] P_{\mathbf{\beta}} \\
&+\frac{N-k}{N}\limfunc{Tr}\nolimits_{r_{k+1}}\sum_{j=1}^{k}  P_{\mathbf{\alpha}} \left[V_{N,\omega }\left( r_{j}-r_{k+1}\right) , \; \tilde{\gamma}_{N,\omega }^{(k+1)}\right] P_{\mathbf{\beta}} 
\end{aligned}
\end{equation}

Let
$$\text{I} = - i\sum_{j=1}^k \tr J^{(k)}[-\Delta_{x_j}, P_{\mathbf{\alpha}} \tilde \gamma_{N,\omega}^{(k)} P_{\mathbf{\beta}}]$$
\begin{equation}
\label{E:cpct7}
\text{II} = -\omega i \sum_{j=1}^k \tr J^{(k)} [ -\partial_{z_j}^2 + z_j^2, P_{\mathbf{\alpha}} \tilde \gamma_{N,\omega}^{(k)} P_{\mathbf{\beta}}]
\end{equation}
$$\text{III} = -i N^{-1} \sum_{1\leq i < j \leq k} \tr J^{(k)} P_{\mathbf{\alpha}} [V_{N,\omega}(r_i-r_j), \tilde \gamma_{N,\omega}^{(k)} ] P_{\mathbf{\beta}}$$
$$\text{IV} = - i\frac{N-k}{N} \sum_{j=1}^k \tr J^{(k)} P_{\mathbf{\alpha}} [ V_{N,\omega}(r_j-r_{k+1}), \tilde \gamma_{N,\omega}^{(k+1)}] P_{\mathbf{\beta}}$$

Then it follows from \eqref{E:cpct6} that
\begin{equation}
\label{E:cpct50}
\partial_t \tr J^{(k)} P_{\mathbf{\alpha}} \tilde \gamma_{N,\omega}^{(k)} P_{\mathbf{\beta}} = \text{I} + \text{II} + \text{III} + \text{IV}
\end{equation}

First, consider $\text{I}$.  Applying Lemma \ref{L:trace-of-tp-kernel} and then integration by parts, we obtain
\begin{align*}
\text{I} &= i\sum_{j=1}^k \left( \la J^{(k)} \Delta_{x_j} P_{\mathbf{\alpha}} \psi, P_{\mathbf{\beta}} \psi \ra_{\mathbf{r}_k} - \la J^{(k)} P_{\mathbf{\alpha}} \psi, P_{\mathbf{\beta}} \Delta_{x_j}\psi\ra_{\mathbf{r}_k} \right) \\
&= i\sum_{j=1}^k \left( \la J^{(k)} \Delta_{x_j} P_{\mathbf{\alpha}} \psi,   P_{\mathbf{\beta}} \psi \ra_{\mathbf{r}_k} - \la \Delta_{x_j} J^{(k)} P_{\mathbf{\alpha}} \psi , P_{\mathbf{\beta}} \psi \ra_{\mathbf{r}_k} \right)
\end{align*}
Hence
\begin{equation}
\label{E:cpct51}
|\text{I}| \leq \sum_{j=1}^k ( \|J^{(k)} \Delta_{x_j} \|_{\op} + \|\Delta_{x_j} J^{(k)} \|_{\op}) \|P_{\mathbf{\alpha}} \psi \|_{L^2(\mathbb{R}^{3N})} \|P_{\mathbf{\beta}} \psi \|_{L^2(\mathbb{R}^{3N})} \leq C_{k,J^{(k)}} 
\end{equation}
where in the last step we applied the energy estimate.  

Now, consider \text{II}.  When $\mathbf{\alpha}=0$ and $\mathbf{\beta}=0$, we use that
$$\text{II} = -\omega i \sum_{j=1}^k \tr J^{(k)} [ 1-\partial_{z_j}^2 + z_j^2, P_{\mathbf{\alpha}} \tilde \gamma_{N,\omega}^{(k)} P_{\mathbf{\beta}}]=0$$
Otherwise, we proceed directly from \eqref{E:cpct7}, applying Lemma \ref{L:trace-of-tp-kernel} and integration by parts to obtain ($H_j = -\partial_{z_j}^2+z_j^2$)
\begin{align*}
\text{II} 
&= \omega i \sum_{j=1}^k \la J^{(k)} H_j P_{\mathbf{\alpha}} \psi, P_{\mathbf{\beta}} \psi \ra - \la J^{(k)} P_{\mathbf{\alpha}} \psi, H_j P_{\mathbf{\beta}} \psi \ra \\
&= \omega i \sum_{j=1}^k \la J^{(k)} H_j P_{\mathbf{\alpha}} \psi, P_{\mathbf{\beta}} \psi \ra - \la H_j J^{(k)} P_{\mathbf{\alpha}} \psi, P_{\mathbf{\beta}} \psi \ra
\end{align*}
Hence
$$|\text{II}| \lesssim \omega \sum_{j=1}^k (\| J^{(k)} H_j\|_{\op} + \|H_j J^{(k)}\|_{\op}) \| P_{\mathbf{\alpha}} \psi \|_{L^2(\mathbb{R}^{3N})} \| P_{\mathbf{\beta}} \psi \|_{L^2(\mathbb{R}^{3N})}$$
By the energy estimates, 
\begin{equation}
\label{E:cpct52}
\text{II} 
\begin{cases}
=0 & \text{if } \mathbf{\alpha}=0 \text{ and } \mathbf{\beta}=0 \\
\lesssim C_{k, J^{(k)}}  \; \omega^{1- \frac12 |\mathbf{\alpha}| - \frac12 |\mathbf{\beta}|} & \text{otherwise}
\end{cases}
\end{equation}

Now, consider $\text{III}$.  
$$\text{III}=-iN^{-1} \sum_{1\leq i<j\leq k} \la J^{(k)} P_{\mathbf{\alpha}} V_{N,\omega}(r_i-r_j) \psi, P_{\mathbf{\beta}} \psi \ra - \la J^{(k)} P_{\mathbf{\alpha}} \psi, P_{\mathbf{\beta}} V_{N,\omega}(r_i-r_j) \psi \ra$$
$$=-iN^{-1} \sum_{1\leq i<j\leq k} \la J^{(k)} P_{\mathbf{\alpha}} V_{N,\omega}(r_i-r_j) \psi, P_{\mathbf{\beta}} \psi \ra - \la  P_{\mathbf{\alpha}} \psi, J^{(k)} P_{\mathbf{\beta}} V_{N,\omega}(r_i-r_j) \psi \ra$$
Let $L_i = (1-\Delta_{r_i})^{1/2}$ and 
$$W_{ij} = L_i^{-1}L_j^{-1}V_{N,\omega}(r_i-r_j) L_i^{-1}L_j^{-1} \,.$$  
Then
$$\text{III} = -iN^{-1} \sum_{1\leq i<j\leq k}  \la J^{(k)} P_{\mathbf{\alpha}} L_iL_j W_{ij} L_iL_j \psi, P_{\mathbf{\beta}} \psi \ra  - \la  P_{\mathbf{\alpha}} \psi, J^{(k)} P_{\mathbf{\beta}} L_iL_j W_{ij} L_iL_j \psi \ra$$
Hence
$$|\text{III}| \lesssim 
\begin{aligned}[t]
&N^{-1}\| J^{(k)} L_iL_j \|_{\op} \|W_{ij}\|_{\op} \|L_iL_j \psi\|_{L^2(\mathbb{R}^{3N})} \|P_{\mathbf{\beta}} \psi \|_{L^2(\mathbb{R}^{3N})} \\
&+ N^{-1}\|P_{\mathbf{\alpha}} \psi \|_{L^2(\mathbb{R}^{3N})} \|J^{(k)} L_iL_j \|_{\op} \|W_{ij}\|_{\op} \|L_iL_j \psi \|_{L^2(\mathbb{R}^{3N})}
\end{aligned}
$$
By Lemma \ref{Lemma:ESYSoblevLemma}, $\| W_{ij}\|_{\op} \lesssim \|V_{N,\omega} \|_{L^1} = \|V\|_{L^1}$ (independent of $N$, $\omega$), and hence the energy estimates imply that
\begin{equation}
\label{E:cpct53}
| \text{III} | \lesssim C_{k,J^{(k)}} \; N^{-1}
\end{equation}
Now consider $\text{IV}$.  
$$\text{IV} = -i\frac{N-k}{N} \sum_{j=1}^k \left( \la J^{(k)} P_{\mathbf{\alpha}} V_{N,\omega}(r_j-r_{k+1}) \psi, P_{\mathbf{\beta}} \psi \ra - \la  J^{(k)} P_{\mathbf{\alpha}} \psi, P_{\mathbf{\beta}} V_{N,\omega}(r_j-r_{k+1}) \psi \ra \right)$$
Then, since $J^{(k)}L_{k+1} = L_{k+1} J^{(k)}$,
$$\text{IV} = 
\begin{aligned}[t] 
&-i\frac{N-k}{N} \sum_{j=1}^k \la J^{(k)} L_jP_{\mathbf{\alpha}}  W_{j(k+1)} L_jL_{k+1} \psi, P_{\mathbf{\beta}} L_{k+1}\psi \ra \\
& -i\frac{N-k}{N} \sum_{j=1}^k \la  L_j   J^{(k)} P_{\mathbf{\alpha}} L_{k+1}\psi, P_{\mathbf{\beta}}  W_{j(k+1)} L_j L_{k+1} \psi \ra \end{aligned}
$$
Estimating yields
$$ | \text{IV} | \lesssim \sum_{j=1}^k ( \| J^{(k)} L_j \|_{\op} + \| L_j J^{(k)} \|_{\op}) \|W_{j(k+1)} \|_{\op} \| L_jL_{k+1} \psi \|_{L^2(\mathbb{R}^{3N})} \|L_{k+1} \psi \|_{L^2(\mathbb{R}^{3N})}$$
By \eqref{E:e-2},
\begin{equation}
\label{E:cpct54}
| \text{IV} | \lesssim C_{k,J^{(k)}}
\end{equation}
Integrating \eqref{E:cpct50} from $t_1$ to $t_2$ and applying the bounds obtained in \eqref{E:cpct51}, \eqref{E:cpct52}, \eqref{E:cpct53}, and \eqref{E:cpct54}, we obtain \eqref{E:cpct2}.

Finally, we proceed to prove \eqref{E:cpct3}.  We have, by Lemma \ref{Lemma:ESYSoblevLemma},
\begin{align*}
\indentalign | \tr J^{(k)} P_{\mathbf{\alpha}} \tilde{\gamma}_{N,\omega}^{(k)} P_{\mathbf{\beta}}(t_2) - \tr J^{(k)} P_{\mathbf{\alpha}} \tilde{\gamma}_{N,\omega}^{(k)} P_{\mathbf{\beta}}(t_1)|  \\
&\leq 2 \sup_t | \la J^{(k)}P_{\mathbf{\alpha}} \tilde \psi_{N,\omega}(t), P_{\mathbf{\beta}} \tilde \psi_{N,\omega}(t) \ra_{\mathbf{r}_k} | \\
&\lesssim \| J^{(k)} \|_{\op} \| P_{\mathbf{\alpha}} \tilde \psi_{N,\omega}(t)\|_{L^2(\mathbb{R}^{3N})} \| P_{\mathbf{\beta}} \tilde \psi_{N,\omega}(t)\|_{L^2(\mathbb{R}^{3N})}  \\
 &\lesssim \omega^{-\frac12 |\mathbf{\alpha}|- \frac12 |\mathbf{\beta}|}
\end{align*}
where in the last step we applied \eqref{E:e-3}.

\end{proof}


According to Corollary \ref{Corollary:LimitMustBeAProduct}, the study of the
limit point of $\Gamma _{N,\omega }(t)=\left\{ \tilde{\gamma}_{N,\omega
}^{(k)}\right\} _{k=1}^{N}$ is directly related to the sequence $\Gamma
_{x,N,\omega }(t)=\left\{ \tilde{\gamma}_{x,N,\omega }^{(k)}=\limfunc{Tr}_{z}%
\tilde{\gamma}_{N,\omega }^{(k)}\right\} _{k=1}^{N}\in \oplus _{k\geqslant
1}C\left( \left[ 0,T\right] ,\mathcal{L}_{k}^{1}\left( \mathbb{R}%
^{2k}\right) \right) .$ We will do so in Section \ref{Section:Convergence of
The Infinite Hierarchy}. We end this section on compactness by proving that $%
\Gamma _{x,N,\omega }(t)$ is compact with respect to the two dimensional
version of the product topology $\tau _{prod}$ used in Theorem \ref%
{Theorem:Compactness of the scaled marginal density}. This proof is not as
delicate as the proof of Theorem \ref{Theorem:Compactness of the scaled
marginal density} because we do not need to deal with $\infty -\infty $ here.

\begin{theorem}
\label{Theorem:Compactness of the x-marginal density}The sequence 
\begin{equation*}
\Gamma _{x,N,\omega }(t)=\left\{ \tilde{\gamma}_{x,N,\omega }^{(k)}=\limfunc{%
Tr}\nolimits_{z}\tilde{\gamma}_{N,\omega }^{(k)}\right\} _{k=1}^{N}\in
\bigoplus_{k\geqslant 1}C\left( \left[ 0,T\right] ,\mathcal{L}_{k}^{1}\left( 
\mathbb{R}^{2k}\right) \right) .
\end{equation*}%
is compact with respect to the two dimensional version of the product
topology $\tau _{prod}$ used in Theorem \ref{Theorem:Compactness of the
scaled marginal density}.
\end{theorem}

\begin{proof}
Similar to Theorem \ref{Theorem:Compactness of the scaled marginal density},
we show that for every observable $J_{x}^{(k)}$ from a dense subset of $%
\mathcal{K}\left( L^{2}\left( \mathbb{R}^{2k}\right) \right) $ and for every 
$\varepsilon >0,$ $\exists \delta (J_{x}^{(k)},\varepsilon )$ s.t. $\forall
t_{1},t_{2}\in \left[ 0,T\right] $ with $\left\vert t_{1}-t_{2}\right\vert
\leqslant \delta ,$ we have%
\begin{equation*}
\sup_{N,\omega }\left\vert \limfunc{Tr}J_{x}^{(k)}\left( \tilde{\gamma}%
_{x,N,\omega }^{(k)}\left( t_{1}\right) -\tilde{\gamma}_{x,N,\omega
}^{(k)}\left( t_{2}\right) \right) \right\vert \leqslant \varepsilon .
\end{equation*}%
We utilize the observables $J_{x}^{(k)}\in \mathcal{K}\left( L^{2}\left( 
\mathbb{R}^{2k}\right) \right) $ which satisfy 
\begin{equation*}
\left\Vert \left\langle \nabla _{x_{i}}\right\rangle \left\langle \nabla
_{x_{j}}\right\rangle J_{x}^{(k)}\left\langle \nabla _{x_{i}}\right\rangle
^{-1}\left\langle \nabla _{x_{j}}\right\rangle ^{-1}\right\Vert _{\func{op}%
}+\left\Vert \left\langle \nabla _{x_{i}}\right\rangle ^{-1}\left\langle
\nabla _{x_{j}}\right\rangle ^{-1}J_{x}^{(k)}\left\langle \nabla
_{x_{i}}\right\rangle \left\langle \nabla _{x_{j}}\right\rangle \right\Vert
_{\func{op}}<\infty .
\end{equation*}%
Here we choose similar but different observables from the proof of Theorem %
\ref{Theorem:Compactness of the scaled marginal density} since $\tilde{\gamma%
}_{x,N,\omega }^{(k)}$ acts on $L^{2}\left( \mathbb{R}^{2k}\right) $ instead
of $L^{2}\left( \mathbb{R}^{3k}\right) .$ This seems to make a difference
when we deal with the terms involving $\tilde{\gamma}_{N,\omega }^{(k)}$ or $%
\tilde{\gamma}^{(k)}.$ But $J_{x}^{(k)}$ does nothing on the $z$ variable,
hence 
\begin{eqnarray*}
\left\Vert L_{j}J_{x}^{(k)}L_{j}^{-1}\right\Vert _{\func{op}} &\sim
&\left\Vert \left( \left\langle \nabla _{x_{j}}\right\rangle +\partial
_{z_{j}}\right) J_{x}^{(k)}\frac{1}{\left( \left\langle \nabla
_{x_{j}}\right\rangle +\partial _{z_{j}}\right) }\right\Vert _{\func{op}} \\
&\leqslant &\left\Vert \left\langle \nabla _{x_{j}}\right\rangle J_{x}^{(k)}%
\frac{1}{\left( \left\langle \nabla _{x_{j}}\right\rangle +\partial
_{z_{j}}\right) }\right\Vert _{\func{op}}+\left\Vert J_{x}^{(k)}\frac{%
\partial _{z_{j}}}{\left( \left\langle \nabla _{x_{j}}\right\rangle
+\partial _{z_{j}}\right) }\right\Vert _{\func{op}} \\
&\leqslant &\left\Vert \left\langle \nabla _{x_{j}}\right\rangle
J_{x}^{(k)}\left\langle \nabla _{x_{j}}\right\rangle ^{-1}\right\Vert _{%
\func{op}}+\left\Vert J_{x}^{(k)}\right\Vert _{\func{op}},
\end{eqnarray*}%
i.e. $\Vert L_{j}J_{x}^{(k)}L_{j}^{-1}\Vert _{\func{op}},\Vert
L_{j}^{-1}J_{x}^{(k)}L_{j}\Vert _{\func{op}},\Vert
L_{i}L_{j}J_{x}^{(k)}L_{i}^{-1}L_{j}^{-1}\Vert _{\func{op}}$ and $\Vert
L_{i}^{-1}L_{j}^{-1}J_{x}^{(k)}L_{i}L_{j}\Vert _{\func{op}}$ are all finite.
It is true that $J_{x}^{(k)}$ and the related operators listed are only in $%
\mathcal{L}^{\infty }\left( L^{2}\left( \mathbb{R}^{3k}\right) \right) $,
but this is good enough for our purpose here.

Taking $\limfunc{Tr}_{z}$ on both sides of hierarchy \eqref{hierarchy:BBGKY
hierarchy for scaled marginal densities}, we have that $\tilde{\gamma}%
_{x,N,\omega }^{(k)}$ satisfies the coupled BBGKY hierarchy:%
\begin{eqnarray}
i\partial _{t}\tilde{\gamma}_{x,N,\omega }^{(k)} &=&\sum_{j=1}^{k}\left[
-\triangle _{x_{j}},\tilde{\gamma}_{x,N,\omega }^{(k)}\right] +\frac{1}{N}%
\sum_{i<j}^{k}\limfunc{Tr}\nolimits_{z}\left[ V_{N,\omega }\left(
r_{i}-r_{j}\right) ,\tilde{\gamma}_{N,\omega }^{(k)}\right] 
\label{hierarchy:coupled BBGKY for the x-component} \\
&&+\frac{N-k}{N}\sum_{j=1}^{k}\limfunc{Tr}\nolimits_{x_{k+1}}\limfunc{Tr}%
\nolimits_{z}\left[ V_{N,\omega }\left( r_{j}-r_{k+1}\right) ,\tilde{\gamma}%
_{N,\omega }^{(k+1)}\right] .  \notag
\end{eqnarray}

Assume $t_{1}\leqslant t_{2},$ the above hierarchy yields%
\begin{eqnarray*}
&&\left\vert \limfunc{Tr}J_{x}^{(k)}\left( \tilde{\gamma}_{x,N,\omega
}^{(k)}\left( t_{1}\right) -\tilde{\gamma}_{x,N,\omega }^{(k)}\left(
t_{2}\right) \right) \right\vert  \\
&\leqslant &\sum_{j=1}^{k}\int_{t_{1}}^{t_{2}}\left\vert \limfunc{Tr}%
J_{x}^{(k)}\left[ -\triangle _{x_{j}},\tilde{\gamma}_{x,N,\omega }^{(k)}%
\right] \right\vert dt+\frac{1}{N}\sum_{i<j}^{k}\int_{t_{1}}^{t_{2}}\left%
\vert \limfunc{Tr}J_{x}^{(k)}\left[ V_{N,\omega }\left( r_{i}-r_{j}\right) ,%
\tilde{\gamma}_{N,\omega }^{(k)}\right] \right\vert dt \\
&&+\frac{N-k}{N}\sum_{j=1}^{k}\int_{t_{1}}^{t_{2}}\left\vert \limfunc{Tr}%
J_{x}^{(k)}\left[ V_{N,\omega }\left( r_{j}-r_{k+1}\right) ,\tilde{\gamma}%
_{N,\omega }^{(k+1)}\right] \right\vert dt. \\
&=&\sum_{j=1}^{k}\int_{t_{1}}^{t_{2}}\text{I}\left( t\right) dt+\frac{1}{N}%
\sum_{i<j}^{k}\int_{t_{1}}^{t_{2}}\text{II}\left( t\right) dt+\frac{N-k}{N}%
\sum_{j=1}^{k}\int_{t_{1}}^{t_{2}}\text{III}\left( t\right) dt.
\end{eqnarray*}%
For I, we have
\begin{eqnarray*}
&& \hspace{-0.5in} \left\vert \limfunc{Tr}J_{x}^{(k)}\left[ -\triangle _{x_{j}},\tilde{\gamma}%
_{x,N,\omega }^{(k)}\right] \right\vert  \\
&=&\left\vert \limfunc{Tr}J_{x}^{(k)}\left[ \left\langle \nabla
_{x_{j}}\right\rangle ^{2},\tilde{\gamma}_{x,N,\omega }^{(k)}\right]
\right\vert \text{ (}1\text{ commutes with everything)} \\
&=&\left\vert \limfunc{Tr}\left\langle \nabla _{x_{j}}\right\rangle
^{-1}J_{x}^{(k)}\left\langle \nabla _{x_{j}}\right\rangle ^{2}\tilde{\gamma}%
_{x,N,\omega }^{(k)}\left\langle \nabla _{x_{j}}\right\rangle -\limfunc{Tr}%
\left\langle \nabla _{x_{j}}\right\rangle J_{x}^{(k)}\left\langle \nabla
_{x_{j}}\right\rangle ^{-1}\left\langle \nabla _{x_{j}}\right\rangle \tilde{%
\gamma}_{x,N,\omega }^{(k)}\left\langle \nabla _{x_{j}}\right\rangle
\right\vert  \\
&\leqslant &\left( \left\Vert \left\langle \nabla _{x_{j}}\right\rangle
^{-1}J_{x}^{(k)}\left\langle \nabla _{x_{j}}\right\rangle \right\Vert
_{\op}+\left\Vert \left\langle \nabla _{x_{j}}\right\rangle
J_{x}^{(k)}\left\langle \nabla _{x_{j}}\right\rangle ^{-1}\right\Vert
_{\op}\right) \limfunc{Tr}\left\langle \nabla _{x_{j}}\right\rangle \tilde{%
\gamma}_{x,N,\omega }^{(k)}\left\langle \nabla _{x_{j}}\right\rangle  \\
&\leqslant &C_{J}\limfunc{Tr}\left\langle \nabla _{x_{j}}\right\rangle ^{2}%
\tilde{\gamma}_{N,\omega }^{(k)} \\
&\leqslant &C_{J}\text{ (Corollary \ref{Corollary:Energy Bound for Marginal
Densities}).}
\end{eqnarray*}

for II and III, we have%
\begin{eqnarray*}
\text{II} &=&\left\vert \limfunc{Tr}J_{x}^{(k)}\left[ V_{N,\omega }\left(
r_{i}-r_{j}\right) ,\tilde{\gamma}_{N,\omega }^{(k)}\right] \right\vert  \\
&=&|\limfunc{Tr}L_{i}^{-1}L_{j}^{-1}J_{x}^{(k)}L_{i}L_{j}W_{ij}L_{i}L_{j}%
\tilde{\gamma}_{N,\omega }^{(k)}L_{i}L_{j}-\limfunc{Tr}%
L_{i}L_{j}J_{x}^{(k)}L_{i}^{-1}L_{j}^{-1}L_{i}L_{j}\tilde{\gamma}_{N,\omega
}^{(k)}L_{i}L_{j}W_{ij}| \\
&\leqslant &\left( \left\Vert
L_{i}^{-1}L_{j}^{-1}J_{x}^{(k)}L_{i}L_{j}\right\Vert _{\op}+\left\Vert
L_{i}L_{j}J_{x}^{(k)}L_{i}^{-1}L_{j}^{-1}\right\Vert _{\op}\right) \left\Vert
W_{ij}\right\Vert _{\op}\limfunc{Tr}L_{i}L_{j}\tilde{\gamma}_{N,\omega
}^{(k)}L_{i}L_{j} \\
&\leqslant &C_{J}\text{,}
\end{eqnarray*}%
and similarly,%
\begin{eqnarray*}
\text{III} &=&\left\vert \limfunc{Tr}J_{x}^{(k)}\left[ V_{N,\omega }\left(
r_{j}-r_{k+1}\right) ,\tilde{\gamma}_{N,\omega }^{(k+1)}\right] \right\vert 
\\
&=&|\limfunc{Tr}%
L_{j}^{-1}L_{k+1}^{-1}J_{x}^{(k)}L_{j}L_{k+1}W_{j(k+1)}L_{j}L_{k+1}\tilde{%
\gamma}_{N,\omega }^{(k+1)}L_{j}L_{k+1} \\
&&-\limfunc{Tr}L_{j}L_{k+1}J_{x}^{(k)}L_{j}^{-1}L_{k+1}^{-1}L_{j}L_{k+1}%
\tilde{\gamma}_{N,\omega }^{(k+1)}L_{j}L_{k+1}W_{j(k+1)}| \\
&\leqslant &\left( \left\Vert L_{j}^{-1}J_{x}^{(k)}L_{j}\right\Vert
_{\op}+\left\Vert L_{j}J_{x}^{(k)}L_{j}^{-1}\right\Vert _{\op}\right)
\left\Vert W_{j(k+1)}\right\Vert _{\op}\limfunc{Tr}L_{j}L_{k+1}\tilde{\gamma}%
_{N,\omega }^{(k+1)}L_{j}L_{k+1} \\
&\leqslant &C_{J}.
\end{eqnarray*}%
Up to this point, we have proven uniform in time bounds for I - III,
thus we conclude the compactness of the sequence $\Gamma _{x,N,\omega
}(t)=\left\{ \tilde{\gamma}_{x,N,\omega }^{(k)}\right\} _{k=1}^{N}$. 
\end{proof}

\section{Limit points satisfy GP hierarchy\label{Section:Convergence of The Infinite Hierarchy}}

\begin{theorem}
\label{Theorem:Convergence to the Coupled Gross-Pitaevskii}
Let $\Gamma(t)=\left\{ \tilde{\gamma}^{(k)}\right\} _{k=1}^{\infty }$ be a $N\geqslant \omega ^{v(\beta)+\varepsilon }$ limit point of $\Gamma_{N,\omega }(t)=\left\{ \tilde{\gamma}_{N,\omega }^{(k)}\right\} _{k=1}^{N}$ with respect to the product topology $\tau _{prod}$, then $\left\{ \tilde{\gamma}_{x}^{(k)}=\limfunc{Tr}_{z}\tilde{\gamma}^{(k)}\right\}_{k=1}^{\infty }$ is a solution to the coupled Gross-Pitaevskii hierarchy subject to initial data $\tilde{\gamma}_{x}^{(k)}\left( 0\right) =\left\vert\phi _{0}\right\rangle \left\langle \phi _{0}\right\vert ^{\otimes k}$ with coupling constant $b_{0}=$ $\int V\left( r\right) dr$, which, written in integral form, is
\begin{equation}
\label{hierarchy:coupled Gross-Pitaevskii}
\tilde{\gamma}_{x}^{(k)}=U^{(k)}(t)\tilde{\gamma}_{x}^{(k)}\left( 0\right)-ib_{0}\sum_{j=1}^{k}\int_{0}^{t}U^{(k)}(t-s)\limfunc{Tr}\nolimits_{x_{k+1}}\limfunc{Tr}\nolimits_{z}\left[ \delta \left( r_{j}-r_{k+1}\right) ,\tilde{\gamma}^{(k+1)}\left( s\right) \right] ds,
\end{equation}
where
\begin{equation*}
U^{(k)}=\dprod\limits_{j=1}^{k}e^{it\triangle _{x_{j}}}e^{-it\triangle_{x_{j}^{\prime }}}.
\end{equation*}
\end{theorem}

We prove Theorem \ref{Theorem:Convergence to the Coupled Gross-Pitaevskii} below.  Combining Corollary \ref{Corollary:LimitMustBeAProduct} and Theorem \ref{Theorem:Convergence to the Coupled Gross-Pitaevskii}, we see that $\tilde{\gamma}_{x}^{(k)}$ in fact solves the 2D Gross-Pitaevskii hierarchy with the desired coupling constant $b_{0}\left( \int \left\vert h_{1}\left( z\right)\right\vert ^{4}dz\right) .$

\begin{corollary}
\label{Theorem:Convergence to the 2D Gross-Pitaevskii}Let $\Gamma(t)=\left\{ \tilde{\gamma}^{(k)}\right\} _{k=1}^{\infty }$ be a $N\geqslant \omega ^{v(\beta)+\varepsilon }$ limit point of $\Gamma_{N,\omega }(t)=\left\{ \tilde{\gamma}_{N,\omega }^{(k)}\right\} _{k=1}^{N}$ with respect to the product topology $\tau _{prod}$, then $\left\{ \tilde{\gamma}_{x}^{(k)}=\limfunc{Tr}_{z}\tilde{\gamma}^{(k)}\right\}_{k=1}^{\infty }$ is a solution to the 2D Gross-Pitaevskii hierarchy subject to initial data $\tilde{\gamma}_{x}^{(k)}\left( 0\right) =\left\vert \phi_{0}\right\rangle \left\langle \phi _{0}\right\vert ^{\otimes k}$ with coupling constant $b_{0}\left( \int \left\vert h_{1}\left( z\right)\right\vert ^{4}dz\right) $, which, written in integral form, is
\begin{equation}
\tilde{\gamma}_{x}^{(k)}=U^{(k)}(t)\tilde{\gamma}_{x}^{(k)}\left( 0\right)
-ib_{0}\left( \int \left\vert h_{1}\left( z\right) \right\vert ^{4}dz\right)
\sum_{j=1}^{k}\int_{0}^{t}U^{(k)}(t-s)\limfunc{Tr}\nolimits_{x_{k+1}}\left[
\delta \left( x_{j}-x_{k+1}\right) ,\tilde{\gamma}_{x}^{(k+1)}\left(
s\right) \right] ds.  \label{hierarchy:2D Gross-Pitaevskii}
\end{equation}
\end{corollary}

\begin{proof}
We compute the $k=1$ case explicitly here. Written in kernels, the inhomogeneous term in hierarchy \eqref{hierarchy:coupled Gross-Pitaevskii} is
\begin{align*}
&ib_{0}\int U^{(1)}(t-s)ds\int \delta ( z_{1}-z_{1}^{\prime })dz_{1}dz_{1}^{\prime }\int \delta ( r_{1}-r_{2}) \tilde{\gamma}^{(2)}( r_{1},r_{2},r_{1}^{\prime },r_{2}) dr_{2} \\
&\quad -ib_{0}\int U^{(1)}(t-s)ds\int \delta ( z_{1}-z_{1}^{\prime })dz_{1}dz_{1}^{\prime }\int \delta ( r_{1}^{\prime }-r_{2}) \tilde{\gamma}^{(2)}( r_{1},r_{2},r_{1}^{\prime },r_{2}) dr_{2}
\end{align*}
which, by Corollary \ref{Corollary:LimitMustBeAProduct}, is
\begin{align*}
= & 
\begin{aligned}[t]
ib_{0}\int U^{(1)}(t-s)ds\int \delta ( z_{1}-z_{1}^{\prime })\delta ( r_{1}-r_{2}) \tilde{\gamma}_{x}^{(2)}(x_{1},x_{2},x_{1}^{\prime },x_{2}) &\\
\times h_{1}( z_{1}) h_{1}( z_{2}) h_{1}( z_{1}^{\prime }) & h_{1}(z_{2}) dr_{2}dz_{1}dz_{1}^{\prime } 
\end{aligned}\\
&\quad 
\begin{aligned}
-ib_{0}\int U^{(1)}(t-s)ds\int \delta ( z_{1}-z_{1}^{\prime })\delta ( r_{1}^{\prime }-r_{2}) \tilde{\gamma}_{x}^{(2)}(x_{1},x_{2},x_{1}^{\prime },x_{2}) &\\
\times h_{1}( z_{1})h_{1}( z_{2}) h_{1}( z_{1}^{\prime }) & h_{1}(z_{2}) dr_{2}dz_{1}dz_{1}^{\prime } 
\end{aligned}
\end{align*}
Further simplifications lead to
\begin{align*}
= & \, ib_{0}\int U^{(1)}(t-s)ds\int \delta ( x_{1}-x_{2}) \tilde{\gamma}_{x}^{(2)}( x_{1},x_{2},x_{1}^{\prime },x_{2}) \vert h_{1}( z_{1}) \vert ^{4}dx_{2}dz_{1} \\
&\quad -ib_{0}\int U^{(1)}(t-s)ds\int \delta ( x_{1}^{\prime }-x_{2}) \tilde{\gamma}_{x}^{(2)}( x_{1},x_{2},x_{1}^{\prime },x_{2})\vert h_{1}( z_{1}^{\prime }) \vert^{4}dx_{2}dz_{1}^{\prime }.
\end{align*}
In summary, we have
\begin{align*}
\indentalign ib_{0}\int U^{(1)}(t-s)\limfunc{Tr}\nolimits_{x_{2}}\limfunc{Tr}
\nolimits_{z}\left[ \delta \left( r_{1}-r_{2}\right) ,\tilde{\gamma}
^{(2)}\left( s\right) \right] ds \\
&=ib_{0}\left( \int \left\vert h_{1}\left( z\right) \right\vert
^{4}dz\right) \int U^{(2)}(t-s)\limfunc{Tr}\nolimits_{x_{2}}\left[ \delta
\left( x_{1}-x_{2}\right) ,\tilde{\gamma}_{x}^{(2)}\left( s\right) \right]
ds.
\end{align*}
\end{proof}


\begin{proof}[Proof of Theorem \ref{Theorem:Convergence to the Coupled Gross-Pitaevskii}]

By Theorems \ref{Theorem:Compactness of the scaled marginal density}, \ref{Theorem:Compactness of the x-marginal density}, passing to subsequences if necessary, we have
\begin{equation}
\label{condition:fast convergence} 
\begin{aligned}
\lim_{\substack{ N,\omega \rightarrow \infty  \\ N\geqslant \omega ^{v(\beta
)+\varepsilon }}}\sup_{t}\limfunc{Tr}J^{(k)}\left( \tilde{\gamma}_{N,\omega
}^{(k)}\left( t\right) -\tilde{\gamma}^{(k)}\left( t\right) \right) &=&0,
\quad \forall \; J^{(k)}\in \mathcal{K}\left( L^{2}\left( \mathbb{R}
^{3k}\right) \right) ,  \\
\lim_{\substack{ N,\omega \rightarrow \infty  \\ N\geqslant \omega ^{v(\beta
)+\varepsilon }}}\sup_{t}\limfunc{Tr}J_{x}^{(k)}\left( \tilde{\gamma}
_{x,N,\omega }^{(k)}\left( t\right) -\tilde{\gamma}_{x}^{(k)}\left( t\right)
\right) &=&0,\quad \forall \; J_{x}^{(k)}\in \mathcal{K}\left( L^{2}\left( 
\mathbb{R}^{2k}\right) \right) . 
\end{aligned}
\end{equation}
We establish \eqref{hierarchy:coupled Gross-Pitaevskii} by testing the limit
point against the observables $J_{x}^{(k)}\in \mathcal{K}\left( L^{2}\left( 
\mathbb{R}^{2k}\right) \right) $ as in the proof of Theorem \ref{Theorem:Compactness of the x-marginal density}.  We will prove that the limit point satisfies
\begin{equation}
\label{equality:testing the limit pt with initial data}
\limfunc{Tr}J_{x}^{(k)}\tilde{\gamma}_{x}^{(k)}\left( 0\right) = \limfunc{Tr} J_{x}^{(k)}\left\vert \phi _{0}\right\rangle \left\langle \phi_{0}\right\vert ^{\otimes k}
\end{equation}
and
\begin{equation}
 \label{hierarchy:testing the limit point}
\begin{aligned}
\limfunc{Tr}J_{x}^{(k)}\tilde{\gamma}_{x}^{(k)}\left( t\right) =  &\limfunc{Tr}
J_{x}^{(k)}U^{(k)}\left( t\right) \tilde{\gamma}_{x}^{(k)}\left( 0\right) \\
& -ib_{0}\sum_{j=1}^{k}\int_{0}^{t}\limfunc{Tr}J_{x}^{(k)}U^{(k)}(t-s)\left[
\delta \left( r_{j}-r_{k+1}\right) ,\tilde{\gamma}^{(k+1)}\left( s\right) 
\right] ds. 
\end{aligned}
\end{equation}

To this end, we use the coupled BBGKY hierarchy \eqref{hierarchy:coupled BBGKY for the x-component} satisfied by $\tilde{\gamma}_{x,N,\omega }^{(k)}$, which, written in the form needed here, is
\begin{align*}
\limfunc{Tr}J_{x}^{(k)}\tilde{\gamma}_{x,N,\omega }^{(k)}\left( t\right) = &
\limfunc{Tr}J_{x}^{(k)}U^{(k)}\left( t\right) \tilde{\gamma}_{x,N,\omega
}^{(k)}\left( 0\right) \\
&-\frac{i}{N}\sum_{i<j}^{k}\int_{0}^{t}\limfunc{Tr}
J_{x}^{(k)}U^{(k)}\left( t-s\right) \left[ V_{N,\omega }\left(
r_{i}-r_{j}\right) ,\tilde{\gamma}_{N,\omega }^{(k)}\left( s\right) \right]
ds \\
&-i\left( \frac{N-k}{N}\right) \sum_{j=1}^{k}\int_{0}^{t}\limfunc{Tr}
J_{x}^{(k)}U^{(k)}\left( t-s\right) \left[ V_{N,\omega }\left(
r_{j}-r_{k+1}\right) ,\tilde{\gamma}_{N,\omega }^{(k+1)}\left( s\right) 
\right] ds \\
=&A-\frac{i}{N}\sum_{i<j}^{k}B-i\left( 1-\frac{k}{N}\right) \sum_{j=1}^{k}D.
\end{align*}
By \eqref{condition:fast convergence}, we know
\begin{eqnarray*}
\lim_{\substack{ N,\omega \rightarrow \infty  \\ N\geqslant \omega ^{v(\beta
)+\varepsilon }}}\limfunc{Tr}J_{x}^{(k)}\tilde{\gamma}_{x,N,\omega
}^{(k)}\left( t\right) &=&\limfunc{Tr}J_{x}^{(k)}\tilde{\gamma}
_{x}^{(k)}\left( t\right) , \\
\lim_{\substack{ N,\omega \rightarrow \infty  \\ N\geqslant \omega ^{v(\beta
)+\varepsilon }}}\limfunc{Tr}J_{x}^{(k)}U^{(k)}\left( t\right) \tilde{\gamma}
_{x,N,\omega }^{(k)}\left( 0\right) &=&\limfunc{Tr}J_{x}^{(k)}U^{(k)}\left(
t\right) \tilde{\gamma}_{x}^{(k)}\left( 0\right) .
\end{eqnarray*}
By the argument that appears between Theorem 1 and Corollary 1 in \cite{LiebAndSeiringer}, we know that assumption (b) in Theorem \ref{Theorem:3D->2D BEC (Nonsmooth)},
\begin{equation*}
\tilde{\gamma}_{N,\omega }^{(1)}\left( 0\right) \rightarrow \left\vert \phi
_{0}\otimes h_{1}\right\rangle \left\langle \phi _{0}\otimes h_{1}\right\vert \,,
\quad \text{strongly in trace norm}\,,
\end{equation*}
in fact implies 
\begin{equation*}
\tilde{\gamma}_{N,\omega }^{(k)}\left( 0\right) \rightarrow \left\vert \phi
_{0}\otimes h_{1}\right\rangle \left\langle \phi _{0}\otimes
h_{1}\right\vert ^{\otimes k} \,,
\quad \text{strongly in trace norm}\,.
\end{equation*}
Thus we have tested relation \eqref{equality:testing the limit pt with initial
data}, the left-hand side of \eqref{hierarchy:testing the limit point}, and the first
term on the right-hand side of \eqref{hierarchy:testing the limit point} for the limit
point. We are left to prove that
\begin{eqnarray*}
\lim_{\substack{ N,\omega \rightarrow \infty  \\ N\geqslant \omega ^{v(\beta
)+\varepsilon }}}\frac{B}{N} &=&0, \\
\lim_{\substack{ N,\omega \rightarrow \infty  \\ N\geqslant \omega ^{v(\beta
)+\varepsilon }}}\left( 1-\frac{k}{N}\right) D &=&b_{0}\int_{0}^{t}\limfunc{
Tr}J_{x}^{(k)}U^{(k)}(t-s)\left[ \delta \left( r_{j}-r_{k+1}\right) ,\tilde{
\gamma}^{(k+1)}\left( s\right) \right] ds.
\end{eqnarray*}
First of all, we can use an argument similar to the estimate of $\text{III}$ and $\text{IV}$ in the proof of Theorem \ref{Theorem:Compactness of the scaled marginal density} to show the boundedness of $\left\vert B\right\vert $ and $
\left\vert D\right\vert $ for every finite time $t$. In fact, noticing that $
U^{(k)}$ commutes with Fourier multipliers, we have 
\begin{eqnarray*}
\left\vert B\right\vert &\leqslant &\int_{0}^{t}\left\vert \limfunc{Tr}
J_{x}^{(k)}U^{(k)}\left( t-s\right) \left[ V_{N,\omega }\left(
r_{i}-r_{j}\right) ,\tilde{\gamma}_{N,\omega }^{(k)}\left( s\right) \right]
\right\vert ds \\
&=&\int_{0}^{t}ds|\limfunc{Tr}
L_{i}^{-1}L_{j}^{-1}J_{x}^{(k)}L_{i}L_{j}U^{(k)}\left( t-s\right)
W_{ij} L_iL_j\tilde{\gamma}_{N,\omega }^{(k)}\left(
s\right) L_{i}L_{j} \\
&&-\limfunc{Tr}L_{i}L_{j}J_{x}^{(k)}L_{i}^{-1}L_{j}^{-1}U^{(k)}\left(
t-s\right) L_{i}L_{j}\tilde{\gamma}_{N,\omega }^{(k)}\left( s\right)
L_{i}L_{j}W_{ij}| \\
&\leqslant &\int_{0}^{t}ds\left\Vert
L_{i}^{-1}L_{j}^{-1}J_{x}^{(k)}L_{i}L_{j}\right\Vert _{\op}\left\Vert
U^{(k)}\right\Vert _{\op}\left\Vert W_{ij} \right\Vert \limfunc{Tr}
L_{i}^{2}L_{j}^{2}\tilde{\gamma}_{N,\omega }^{(k)}\left( s\right) \\
&&+\int_{0}^{t}ds\left\Vert
L_{i}L_{j}J_{x}^{(k)}L_{i}^{-1}L_{j}^{-1}\right\Vert _{\op}\left\Vert
U^{(k)}\right\Vert _{\op}\left\Vert W_{ij} \right\Vert \limfunc{Tr}
L_{i}^{2}L_{j}^{2}\tilde{\gamma}_{N,\omega }^{(k)}\left( s\right) \\
&\leqslant &C_{J}t.
\end{eqnarray*}
Hence
\begin{equation*}
\lim_{\substack{ N,\omega \rightarrow \infty  \\ N\geqslant \omega ^{v(\beta
)+\varepsilon }}}\frac{B}{N}=\lim_{\substack{ N,\omega \rightarrow \infty 
\\ N\geqslant \omega ^{v(\beta )+\varepsilon }}}\frac{kD}{N}=0.
\end{equation*}
To prove
\begin{equation}
\lim_{\substack{ N,\omega \rightarrow \infty  \\ N\geqslant \omega ^{v(\beta
)+\varepsilon }}}D=\int_{0}^{t}\limfunc{Tr}J_{x}^{(k)}U^{(k)}(t-s)\left[
\delta \left( r_{j}-r_{k+1}\right) ,\tilde{\gamma}^{(k+1)}\left( s\right) 
\right] ds,  \label{limit:converges to delta function}
\end{equation}
we need Lemma \ref{Lemma:ComparingDeltaFunctions} (stated and proved in Appendix \ref{A:Sobolev}) which compares the $\delta -$function and its  approximation.  We choose a probability measure $\rho \in L^{1}\left( \mathbb{R}^{3}\right) $ and define $\rho _{\alpha
}\left( r\right) =\alpha ^{-3}\rho \left( \frac{r}{\alpha }\right) .$ In
fact, $\rho $ can be the square of any 3D Hermite function. Write $
J_{s-t}^{(k)}=J_{x}^{(k)}U^{(k)}\left( t-s\right) $, we then have 
\begin{align*}
\indentalign \left\vert \limfunc{Tr}J_{x}^{(k)}U^{(k)}\left( t-s\right) \left(
V_{N,\omega }\left( r_{j}-r_{k+1}\right) \tilde{\gamma}_{N,\omega
}^{(k+1)}\left( s\right) -b_{0}\delta \left( r_{j}-r_{k+1}\right) \tilde{
\gamma}^{(k+1)}\left( s\right) \right) \right\vert \\
&\leqslant \left\vert \limfunc{Tr}J_{s-t}^{(k)}\left( V_{N,\omega }\left(
r_{j}-r_{k+1}\right) -b_{0}\delta \left( r_{j}-r_{k+1}\right) \right) \tilde{
\gamma}_{N,\omega }^{(k+1)}\left( s\right) \right\vert \\
&\quad +b_{0}\left\vert \limfunc{Tr}J_{s-t}^{(k)}\left( \delta \left(
r_{j}-r_{k+1}\right) -\rho _{\alpha }\left( r_{j}-r_{k+1}\right) \right) 
\tilde{\gamma}_{N,\omega }^{(k+1)}\left( s\right) \right\vert \\
&\quad +b_{0}\left\vert \limfunc{Tr}J_{s-t}^{(k)}\rho _{\alpha }\left(
r_{j}-r_{k+1}\right) \left( \tilde{\gamma}_{N,\omega }^{(k+1)}\left(
s\right) -\tilde{\gamma}^{(k+1)}\left( s\right) \right) \right\vert \\
&\quad +b_{0}\left\vert \limfunc{Tr}J_{s-t}^{(k)}\left( \rho _{\alpha }\left(
r_{j}-r_{k+1}\right) -\delta \left( r_{j}-r_{k+1}\right) \right) \tilde{
\gamma}^{(k+1)}\left( s\right) \right\vert \\
&=\text{I}+\text{II}+\text{III}+\text{IV}
\end{align*}

We take care of $\text{I}$ first because it is a term which requires $N>\omega ^{\frac{1}{2\beta }-\frac{1}{2}}$. Write $V_{\omega }(r)=\frac{1}{\sqrt{\omega 
}}V(x,\frac{z}{\sqrt{\omega }}),$ we have $V_{N,\omega }=\left( N\sqrt{
\omega }\right) ^{3\beta }V_{\omega }(\left( N\sqrt{\omega }\right) ^{\beta
}r)$, Lemma \ref{Lemma:ComparingDeltaFunctions} then yields 
\begin{eqnarray*}
\text{I} &\leqslant &\frac{Cb_{0}}{\left( N\sqrt{\omega }\right) ^{\beta \kappa }}
\left( \int V_{\omega }(r)\left\vert r\right\vert ^{\kappa }dr\right) \\
&&\times \left( \left\Vert L_{j}J_{x}^{(k)}L_{j}^{-1}\right\Vert
_{\op}+\left\Vert L_{j}^{-1}J_{x}^{(k)}L_{j}\right\Vert _{\op}\right) \limfunc{
Tr}L_{j}L_{k+1}\tilde{\gamma}_{N,\omega }^{(k+1)}\left( s\right) L_{j}L_{k+1}
\\
&=&C_{J}\frac{\left( \int V_{\omega }(r)\left\vert r\right\vert ^{\kappa
}dr\right) }{\left( N\sqrt{\omega }\right) ^{\beta \kappa }}.
\end{eqnarray*}
Notice that $\left( \int V_{\omega }(r)\left\vert r\right\vert ^{\kappa
}dr\right) $ grows like $\left( \sqrt{\omega }\right) ^{\kappa }$, so $
I\leqslant C_{J}\left( \frac{\left( \sqrt{\omega }\right) ^{1-\beta }}{
N^{\beta }}\right) ^{\kappa }$ which converges to zero as $N,\omega
\rightarrow \infty $ in the way that $N\geqslant \omega ^{\frac{1}{2\beta }-
\frac{1}{2}+\varepsilon }.$ More precisely, 
\begin{equation*}
\lim_{\substack{ N,\omega \rightarrow \infty  \\ N\geqslant \omega ^{v(\beta
)+\varepsilon }}}I=0.
\end{equation*}
So we have handled $\text{I}$.

For $\text{II}$ and $\text{IV}$, we have
\begin{eqnarray*}
\text{II} &\leqslant &Cb_{0}\alpha ^{\kappa }\left( \left\Vert
L_{j}J_{x}^{(k)}L_{j}^{-1}\right\Vert _{\op}+\left\Vert
L_{j}^{-1}J_{x}^{(k)}L_{j}\right\Vert _{\op}\right) \limfunc{Tr}L_{j}L_{k+1}
\tilde{\gamma}_{N,\omega }^{(k+1)}\left( s\right) L_{j}L_{k+1}\text{ (Lemma 
\ref{Lemma:ComparingDeltaFunctions})} \\
&\leqslant &C_{J}\alpha ^{\kappa }\text{ (Corollary \ref{Corollary:Energy
Bound for Marginal Densities})} \\
\text{IV} &\leqslant &Cb_{0}\alpha ^{\kappa }\left( \left\Vert
L_{j}J_{x}^{(k)}L_{j}^{-1}\right\Vert _{\op}+\left\Vert
L_{j}^{-1}J_{x}^{(k)}L_{j}\right\Vert _{\op}\right) \limfunc{Tr}L_{j}L_{k+1}
\tilde{\gamma}^{(k+1)}\left( s\right) L_{j}L_{k+1}\text{ (Lemma \ref
{Lemma:ComparingDeltaFunctions})} \\
&\leqslant &C_{J}\alpha ^{\kappa }\text{ (Corollary \ref
{Corollary:LimitMustBeAProduct})}
\end{eqnarray*}
which converges to $0$ as $\alpha \rightarrow 0$, uniformly in $N,\omega .$

For $\text{III}$,
\begin{eqnarray*}
\text{III} &\leqslant &b_{0}\left\vert \limfunc{Tr}J_{s-t}^{(k)}\rho _{\alpha
}\left( r_{j}-r_{k+1}\right) \frac{1}{1+\varepsilon L_{k+1}}\left( \tilde{
\gamma}_{N,\omega }^{(k+1)}\left( s\right) -\tilde{\gamma}^{(k+1)}\left(
s\right) \right) \right\vert \\
&&+b_{0}\left\vert \limfunc{Tr}J_{s-t}^{(k)}\rho _{\alpha }\left(
r_{j}-r_{k+1}\right) \frac{\varepsilon L_{k+1}}{1+\varepsilon L_{k+1}}\left( 
\tilde{\gamma}_{N,\omega }^{(k+1)}\left( s\right) -\tilde{\gamma}
^{(k+1)}\left( s\right) \right) \right\vert .
\end{eqnarray*}
The first term in the above estimate goes to zero as $N,\omega \rightarrow
\infty $ for every $\varepsilon >0$, since we have assumed condition \eqref
{condition:fast convergence} and $J_{s-t}^{(k)}\rho _{\alpha }\left(
r_{j}-r_{k+1}\right) \left( 1+\varepsilon L_{k+1}\right) ^{-1}$ is a compact
operator. Due to the energy bounds on $\tilde{\gamma}_{N,\omega }^{(k+1)}$
and $\tilde{\gamma}^{(k+1)}$, the second term tends to zero as $\varepsilon
\rightarrow 0$, uniformly in $N$.

Combining the estimates for $\text{I}$-$\text{IV}$, we have justified limit \eqref
{limit:converges to delta function}.  Hence, we have obtained Theorem \ref
{Theorem:Convergence to the Coupled Gross-Pitaevskii}.

\end{proof}

\section{Uniqueness of the 2D GP hierarchy\label{Appendix: Uniqueness of 2D GP}}

For completeness, we discuss the uniqueness theory of the 2D Gross-Pitaevskii hierarchy. To be specific, we have the following theorem.

\begin{theorem}[{\cite[Theorem 3]{ChenAnisotropic}}]
\label{Theorem:Uniqueness of 2D GP} Define the
collision operator $B_{j,k+1}$ by
\begin{equation*}
B_{j,k+1}\gamma _{x}^{(k+1)}=\limfunc{Tr}\nolimits_{k+1}\left[ \delta \left(
x_{j}-x_{k+1}\right) ,\gamma _{x}^{(k+1)}\right] .
\end{equation*}
Suppose that $\left\{ \gamma _{x}^{(k)}\right\} _{k=1}^{\infty }$ solves the 2D
constant coefficient Gross-Pitaevskii hierarchy 
\begin{equation}
i\partial _{t}\gamma _{x}^{(k)}+\sum_{j=1}^{k}\left[ -\triangle
_{x_{j}},\gamma _{x}^{(k)}\right] =c_{0}\sum_{j=1}^{k}B_{j,k+1}\left( \gamma
_{x}^{(k+1)}\right) ,  \label{hierarchy:2D GP in appendix, general coupling}
\end{equation}
subject to zero initial data and the space-time bound
\begin{equation}
\int_{0}^{T}\left\Vert \prod_{j=1}^{k}\left( \left\vert \nabla _{x
_{j}}\right\vert ^{\frac{1}{2}}\left\vert \nabla _{x_{j}^{\prime
}}\right\vert ^{\frac{1}{2}}\right) B_{j,k+1}\gamma _{x}^{(k+1)}(t,\mathbf{
\cdot };\mathbf{\cdot })\right\Vert _{L^{2}(\mathbb{R}^{2k}\times \mathbb{R}
^{2k})}dt\leqslant C^{k}  \label{Condition:2D Space-Time Bound}
\end{equation}
for some $C>0$ and all $1\leqslant j\leqslant k.$ Then $\forall k,t\in
\lbrack 0,T]$, 
\begin{equation*}
\left\Vert \prod_{j=1}^{k}\left( \left\vert \nabla _{x
_{j}}\right\vert ^{\frac{1}{2}}\left\vert \nabla _{x_{j}^{\prime
}}\right\vert ^{\frac{1}{2}}\right) \gamma _{x}^{(k)}(t,\mathbf{\cdot };
\mathbf{\cdot })\right\Vert _{L^{2}(\mathbb{R}^{2k}\times \mathbb{R}
^{2k})}=0.
\end{equation*}
\end{theorem}

\begin{proof}
This is the constant coefficient version of \cite[Theorem 3]{ChenAnisotropic}.  W. Beckner obtained the key estimate of this theorem independently in \cite{Beckner}.  Some other estimates of this type can be found in \cite{ChenDie, GM}.  K. Kirpatrick, G. Staffilani and B. Schlein
are the first to obtain uniqueness theorems for 2D Gross-Pitaevskii hierarchies.  One will find their Theorem 7.1 in \cite{Kirpatrick} by replacing $\left\vert \nabla \right\vert ^{\frac{1}{2}}$ by $\left\langle \nabla \right\rangle ^{\frac{1}{2}+\varepsilon }$ in the statement of the above theorem.
\end{proof}

To apply Theorem \ref{Theorem:Uniqueness of 2D GP} to our problem here, it
is necessary to prove that both the known solution to the 2D
Gross-Pitaevskii hierarchy (namely $\left\vert \phi \right\rangle
\left\langle \phi \right\vert ^{\otimes k}$, where $\phi $ solves the 2D
cubic NLS) and the limit obtained from the coupled BBGKY hierarchy \eqref
{hierarchy:coupled BBGKY for the x-component}, satisfy the space-time bound 
\eqref{Condition:2D Space-Time Bound}.  It is easy to see that $\left\vert \phi
\right\rangle \left\langle \phi \right\vert ^{\otimes k}$ verifies the
space-time bound \eqref{Condition:2D Space-Time Bound} because it is part of
the standard procedure of proving well-posedness of the 2D cubic NLS. We use
the following trace theorem to prove the space-time bound \eqref{Condition:2D
Space-Time Bound} for the limit.

\begin{theorem}
[{\cite[Theorem 5.2]{Kirpatrick}}] For every $\alpha <1,$ there is a $C_{\alpha }>0$ such that
\begin{equation*}
\left\Vert \prod_{j=1}^{k}\left( \left\langle \nabla _{x
_{j}}\right\rangle ^{\alpha }\left\langle \nabla _{x_{j}^{\prime
}}\right\rangle ^{\alpha }\right) B_{j,k+1}\gamma _{x}^{(k+1)}\right\Vert
_{L^{2}(\mathbb{R}^{2k}\times \mathbb{R}^{2k})}\leqslant C_{\alpha }\limfunc{
Tr}\left( \prod_{j=1}^{k+1}\left( 1-\triangle _{x_{j}}\right) \right) \gamma
_{x}^{(k+1)}
\end{equation*}
for all nonnegative $\gamma _{x}^{(k+1)}\in \mathcal{L}^{1}\left(
L^{2}\left( \mathbb{R}^{2k}\right) \right) $.
\end{theorem}


We can combine the above theorems so that it is easy to see how they apply to our problem.

\begin{theorem}
\label{Theorem:CombiningChenAndKirpatrick}There is at most one nonnegative
operator sequence 
$$\left\{ \gamma _{x}^{(k)}\right\} _{k=1}^{\infty }\in
\bigoplus _{k\geqslant 1}C\left( \left[ 0,T\right] ,\mathcal{L}_{k}^{1}\left( 
\mathbb{R}^{2k}\right) \right)$$
that solves the 2D Gross-Pitaevskii hierarchy \eqref{hierarchy:2D GP in appendix, general coupling} subject to the
energy condition
\begin{equation*}
\limfunc{Tr}\left( \prod_{j=1}^{k}\left( 1-\triangle _{x_{j}}\right) \right)
\gamma _{x}^{(k)}\leqslant C^{k}.
\end{equation*}
\end{theorem}

\section{Conclusion}

In this paper, by proving the limit of a BBGKY hierarchy whose limit is not
even formally known since it contains $\left( \infty -\infty \right) ,$ we
have rigorously derived the 2D cubic nonlinear Schr\"{o}dinger equation from
a 3D quantum many-body dynamic and we have accurately described the 3D to 2D
phenomenon by establishing the exact emergence of the coupling constant $
\left( \int \left\vert h_{1}(z)\right\vert ^{4}dz\right) $. This is the
first direct rigorous treatment of the 3D to 2D dynamic problem in the literature.

\appendix

\section{Basic operator facts and Sobolev-type lemmas}
\label{A:Sobolev}

\begin{lemma}[{\cite[Lemma A.3]{E-S-Y2}}]
\label{Lemma:ESYSoblevLemma} 
Let $L_{j}=\left( 1-\triangle_{r_{j}}\right) ^{\frac{1}{2}}$.  Then we have
\begin{equation*}
\left\Vert L_{i}^{-1}L_{j}^{-1}V\left( r_{i}-r_{j}\right)
L_{i}^{-1}L_{j}^{-1}\right\Vert _{\op}\leqslant C\left\Vert V\right\Vert
_{L^{1}}.
\end{equation*}
\end{lemma}

\begin{lemma}
\label{Lemma:ComparingDeltaFunctions}Let $\rho \in L^{1}\left( \mathbb{R}
^{3}\right) $ be a probability measure such that $\int_{\mathbb{R}
^{3}}\left\langle r\right\rangle ^{\frac{1}{2}}\rho \left( r\right)
dr<\infty $ and let $\rho _{\alpha }\left( r\right) =\alpha ^{-3}\rho \left( 
\frac{r}{\alpha }\right) .$ Then, for every $\kappa \in \left( 0,1/2\right) $
, there exists $C>0$ s.t.
\begin{align*}
\indentalign \left\vert \limfunc{Tr}J^{(k)}\left( \rho _{\alpha }\left(
r_{j}-r_{k+1}\right) -\delta \left( r_{j}-r_{k+1}\right) \right) \gamma
^{(k+1)}\right\vert \\
&\leqslant C\left( \int \rho \left( r\right) \left\vert
r\right\vert ^{\kappa }dr\right) \alpha ^{\kappa }\left( \left\Vert
L_{j}J^{(k)}L_{j}^{-1}\right\Vert _{\op}+\left\Vert
L_{j}^{-1}J^{(k)}L_{j}\right\Vert _{\op}\right) \limfunc{Tr}L_{j}L_{k+1}\gamma ^{(k+1)}L_{j}L_{k+1}
\end{align*}
for all nonnegative $\gamma ^{(k+1)}\in \mathcal{L}^{1}\left( L^{2}\left( 
\mathbb{R}^{3k+3}\right) \right) .$
\end{lemma}

\begin{proof}
We give a proof by modifying the proof of \cite[Lemma A.2]{Kirpatrick}.  We remark that the range of $\kappa$ is smaller here because we are working in 3D.  It suffices to prove the estimate for $k=1$. We represent $\gamma ^{(2)}$ by 
$\gamma ^{(2)}=\sum_{j}\lambda _{j}\left\vert \varphi _{j}\right\rangle
\left\langle \varphi _{j}\right\vert $, where $\varphi _{j}\in L^{2}\left( 
\mathbb{R}^{6}\right) $ and $\lambda _{j}\geqslant 0.$ We write
\begin{align*}
\indentalign \limfunc{Tr}J^{(1)}\left( \rho _{\alpha }\left( r_{1}-r_{2}\right) -\delta
\left( r_{1}-r_{2}\right) \right) \gamma ^{(2)} \\
&=\sum_{j}\lambda _{j}\left\langle \varphi _{j},J^{(1)}\left( \rho _{\alpha
}\left( r_{1}-r_{2}\right) -\delta \left( r_{1}-r_{2}\right) \right) \varphi
_{j}\right\rangle \\
&=\sum_{j}\lambda _{j}\left\langle \psi _{j},\left( \rho _{\alpha }\left(
r_{1}-r_{2}\right) -\delta \left( r_{1}-r_{2}\right) \right) \varphi
_{j}\right\rangle
\end{align*}
where $\psi _{j}=\left( J^{(1)}\otimes 1\right) \varphi _{j}$. By Parseval,
we find
\begin{align*}
\indentalign \vert \langle \psi _{j},( \rho _{\alpha }(
r_{1}-r_{2}) -\delta ( r_{1}-r_{2}) ) \varphi
_{j}\rangle \vert \\
&=\vert \int \overline{\hat{\psi}}_{j}( \xi _{1},\xi _{2}) 
\hat{\varphi}_{j}( \xi _{1}^{\prime },\xi _{2}^{\prime }) \rho
( r) ( e^{i\alpha r\cdot ( \xi _{1}-\xi _{1}^{\prime
}) }-1) \delta ( \xi _{1}+\xi _{2}-\xi _{1}^{\prime }-\xi
_{2}^{\prime }) drd\xi _{1}d\xi _{2}d\xi _{1}^{\prime }d\xi
_{2}^{\prime }\vert \\
&\leqslant \int \vert \hat{\psi}_{j}( \xi _{1},\xi _{2})
\vert \vert \hat{\varphi}_{j}( \xi _{1}^{\prime },\xi
_{2}^{\prime }) \vert \delta ( \xi _{1}+\xi _{2}-\xi
_{1}^{\prime }-\xi _{2}^{\prime }) \vert \int \rho (
r) ( e^{i\alpha r\cdot ( \xi _{1}-\xi _{1}^{\prime })
}-1) dr\vert d\xi _{1}d\xi _{2}d\xi _{1}^{\prime }d\xi
_{2}^{\prime }.
\end{align*}
Using the inequality that $\forall \kappa \in \left( 0,1\right) $
\begin{eqnarray*}
\left\vert e^{i\alpha r\cdot \left( \xi _{1}-\xi _{1}^{\prime }\right)
}-1\right\vert &\leqslant &\alpha ^{\kappa }\left\vert r\right\vert ^{\kappa
}\left\vert \xi _{1}-\xi _{1}^{\prime }\right\vert ^{\kappa } \\
&\leqslant &\alpha ^{\kappa }\left\vert r\right\vert ^{\kappa }\left(
\left\vert \xi _{1}\right\vert ^{\kappa }+\left\vert \xi _{1}^{\prime
}\right\vert ^{\kappa }\right) ,
\end{eqnarray*}
we get
\begin{align*}
\indentalign \vert \langle \psi _{j},( \rho _{\alpha }(
r_{1}-r_{2}) -\delta ( r_{1}-r_{2}) ) \varphi
_{j}\rangle \vert \\
&\leqslant \alpha ^{\kappa }( \int \rho ( r) \vert
r\vert ^{\kappa }dr) \int \vert \xi _{1}\vert
^{\kappa }\vert \hat{\psi}_{j}( \xi _{1},\xi _{2})
\vert \vert \hat{\varphi}_{j}( \xi _{1}^{\prime },\xi
_{2}^{\prime }) \vert \delta ( \xi _{1}+\xi _{2}-\xi
_{1}^{\prime }-\xi _{2}^{\prime }) d\xi _{1}d\xi _{2}d\xi _{1}^{\prime
}d\xi _{2}^{\prime } \\
&\quad +\alpha ^{\kappa }( \int \rho ( r) \vert r\vert
^{\kappa }dr) \int \vert \xi _{1}^{\prime }\vert ^{\kappa
}\vert \hat{\psi}_{j}( \xi _{1},\xi _{2}) \vert
\vert \hat{\varphi}_{j}( \xi _{1}^{\prime },\xi _{2}^{\prime
}) \vert \delta ( \xi _{1}+\xi _{2}-\xi _{1}^{\prime }-\xi
_{2}^{\prime }) d\xi _{1}d\xi _{2}d\xi _{1}^{\prime }d\xi _{2}^{\prime
} \\
&=\alpha ^{\kappa }( \int \rho ( r) \vert r\vert
^{\kappa }dr) ( \text{I}+\text{II}) .
\end{align*}
The estimate for $\text{I}$ and $\text{II}$ are similar, so we only deal with $\text{I}$ explicitly.
\begin{eqnarray*}
\text{I} &\leqslant &\int \delta \left( \xi _{1}+\xi _{2}-\xi _{1}^{\prime }-\xi
_{2}^{\prime }\right) \frac{\left\langle \xi _{1}\right\rangle \left\langle
\xi _{2}\right\rangle }{\left\langle \xi _{1}^{\prime }\right\rangle
\left\langle \xi _{2}^{\prime }\right\rangle }\left\vert \hat{\psi}
_{j}\left( \xi _{1},\xi _{2}\right) \right\vert \frac{\left\langle \xi
_{1}^{\prime }\right\rangle \left\langle \xi _{2}^{\prime }\right\rangle }{
\left\langle \xi _{1}\right\rangle ^{1-\kappa }\left\langle \xi
_{2}\right\rangle }\left\vert \hat{\varphi}_{j}\left( \xi _{1}^{\prime },\xi
_{2}^{\prime }\right) \right\vert d\xi _{1}d\xi _{2}d\xi _{1}^{\prime }d\xi
_{2}^{\prime } \\
&\leqslant &\varepsilon \int \delta \left( \xi _{1}+\xi _{2}-\xi
_{1}^{\prime }-\xi _{2}^{\prime }\right) \frac{\left\langle \xi
_{1}\right\rangle ^{2}\left\langle \xi _{2}\right\rangle ^{2}}{\left\langle
\xi _{1}^{\prime }\right\rangle ^{2}\left\langle \xi _{2}^{\prime
}\right\rangle ^{2}}\left\vert \hat{\psi}_{j}\left( \xi _{1},\xi _{2}\right)
\right\vert ^{2}d\xi _{1}d\xi _{2}d\xi _{1}^{\prime }d\xi _{2}^{\prime } \\
&&+\frac{1}{\varepsilon }\int \delta \left( \xi _{1}+\xi _{2}-\xi
_{1}^{\prime }-\xi _{2}^{\prime }\right) \frac{\left\langle \xi _{1}^{\prime
}\right\rangle ^{2}\left\langle \xi _{2}^{\prime }\right\rangle ^{2}}{
\left\langle \xi _{1}\right\rangle ^{2\left( 1-\kappa \right) }\left\langle
\xi _{2}\right\rangle ^{2}}\left\vert \hat{\varphi}_{j}\left( \xi
_{1}^{\prime },\xi _{2}^{\prime }\right) \right\vert ^{2}d\xi _{1}d\xi
_{2}d\xi _{1}^{\prime }d\xi _{2}^{\prime } \\
&=&\varepsilon \int \left\langle \xi _{1}\right\rangle ^{2}\left\langle \xi
_{2}\right\rangle ^{2}\left\vert \hat{\psi}_{j}\left( \xi _{1},\xi
_{2}\right) \right\vert ^{2}d\xi _{1}d\xi _{2}\int \frac{1}{\left\langle \xi
_{1}+\xi _{2}-\xi _{2}^{\prime }\right\rangle ^{2}\left\langle \xi
_{2}^{\prime }\right\rangle ^{2}}d\xi _{2}^{\prime } \\
&&\frac{1}{\varepsilon }\int \left\langle \xi _{1}^{\prime }\right\rangle
^{2}\left\langle \xi _{2}^{\prime }\right\rangle ^{2}\left\vert \hat{\varphi}
_{j}\left( \xi _{1}^{\prime },\xi _{2}^{\prime }\right) \right\vert ^{2}d\xi
_{1}^{\prime }d\xi _{2}^{\prime }\int \frac{1}{\left\langle \xi _{1}^{\prime
}+\xi _{2}^{\prime }-\xi _{2}\right\rangle ^{2\left( 1-\kappa \right)
}\left\langle \xi _{2}\right\rangle ^{2}}d\xi _{2} \\
&\leqslant &\varepsilon \left\langle \psi _{j},L_{1}^{2}L_{2}^{2}\psi
_{j}\right\rangle \sup_{\xi }\int_{\mathbb{R}^{3}}\frac{1}{\left\langle \xi
-\eta \right\rangle ^{2}\left\langle \eta \right\rangle ^{2}}d\eta +\frac{1}{
\varepsilon }\left\langle \varphi _{j},L_{1}^{2}L_{2}^{2}\varphi
_{j}\right\rangle \sup_{\xi }\int_{\mathbb{R}^{3}}\frac{1}{\left\langle \xi
-\eta \right\rangle ^{2\left( 1-\kappa \right) }\left\langle \eta
\right\rangle ^{2}}d\eta .
\end{eqnarray*}
When $\kappa \in \lbrack 0,1/2),$ 
\begin{eqnarray*}
\sup_{\xi }\int_{\mathbb{R}^{3}}\frac{1}{\left\langle \xi -\eta
\right\rangle ^{2\left( 1-\kappa \right) }\left\langle \eta \right\rangle
^{2}}d\eta &<&\infty , \\
\sup_{\xi }\int_{\mathbb{R}^{3}}\frac{1}{\left\langle \xi -\eta
\right\rangle ^{2}\left\langle \eta \right\rangle ^{2}}d\eta &<&\infty ,
\end{eqnarray*}
and hence we have (with $\varepsilon =\Vert L_{1}J^{(1)}L_{1}^{-1}\Vert _{\op}^{-1}$),
\begin{align*}
\indentalign \left\vert \limfunc{Tr}J^{(1)}\left( \rho _{\alpha }\left(
r_{1}-r_{2}\right) -\delta \left( r_{1}-r_{2}\right) \right) \gamma
^{(k+1)}\right\vert \\
&\leqslant C\left( \int \rho \left( r\right) \left\vert r\right\vert
^{\kappa }dr\right) \alpha ^{\kappa }\left( \varepsilon \limfunc{Tr}
J^{(1)}L_{1}^{2}L_{2}^{2}J^{(1)}\gamma ^{(2)}+\frac{1}{\varepsilon }\limfunc{
Tr}L_{1}^{2}L_{2}^{2}\gamma ^{(2)}\right) \\
&=C\left( \int \rho \left( r\right) \left\vert r\right\vert ^{\kappa
}dr\right) \alpha ^{\kappa }\left( \varepsilon \limfunc{Tr}
L_{1}^{-1}L_{2}^{-1}J^{(1)}L_{1}L_{1}J^{(1)}L_{1}^{-1}L_{1}L_{2}^{2}\gamma
^{(2)}L_{1}L_{2}+\frac{1}{\varepsilon }\limfunc{Tr}L_{1}^{2}L_{2}^{2}\gamma
^{(2)}\right) \\
&\leqslant C\left( \int \rho \left( r\right) \left\vert r\right\vert
^{\kappa }dr\right) \alpha ^{\kappa }\left( \varepsilon \left\Vert
L_{1}^{-1}J^{(1)}L_{1}\right\Vert _{\op}\left\Vert
L_{1}J^{(1)}L_{1}^{-1}\right\Vert _{\op}+\frac{1}{\varepsilon }\right) 
\limfunc{Tr}L_{1}^{2}L_{2}^{2}\gamma ^{(2)} \\
&\leqslant C\left( \int \rho \left( r\right) \left\vert r\right\vert
^{\kappa }dr\right) \alpha ^{\kappa }\left( \left\Vert
L_{1}^{-1}J^{(1)}L_{1}\right\Vert _{\op}+\left\Vert
L_{1}J^{(1)}L_{1}^{-1}\right\Vert _{\op}\right) \limfunc{Tr}
L_{1}^{2}L_{2}^{2}\gamma ^{(2)}
\end{align*}

\end{proof}

\begin{lemma}[some standard operator inequalities] \quad
\label{L:op-stuff}
\begin{enumerate}
\item \label{I:op-1} Suppose that $A\geq 0$, $P_j=P_j^*$, and $I=P_0+P_1$.  Then $A \leq 2P_0 A P_0 + 2P_1 A P_1$.
\item \label{I:op-2} If $A\geq B \geq 0$, and $AB=BA$, then $A^\alpha \geq B^\alpha$ for any $\alpha \geq 0$.
\item \label{I:op-3} If $A_1 \geq A_2 \geq 0$, $B_1 \geq B_2 \geq 0$ and $A_iB_j = B_jA_i$ for all $1\leq i,j \leq 2$, then $A_1B_1 \geq A_2B_2$.
\item \label{I:op-4}  If $A\geq 0$ and $AB=BA$, then $A^{1/2}B=B A^{1/2}$.
\end{enumerate}
\end{lemma}
\begin{proof}
For (1),  $\|A^{1/2} f\|^2 = \|A^{1/2}(P_0+P_1)f\|^2 \leq 2\|A^{1/2}P_0 f \|^2 + 2\|A^{1/2}P_1 f\|^2$.  The rest are standard facts in operator theory.
\end{proof}

Recall that
$$S^2 = 1- \Delta_x - \omega - \partial_z^2 + \omega^2z^2$$

\begin{lemma}[Estimates with $\omega$-loss]
\label{L:Sobolev-with-loss}
Suppose $f=f(x,z)$.  Then
\begin{gather}
\label{E:Sob1} \| \nabla_r f \|_{L_r^2} \lesssim \omega^{1/2} \|Sf\|_{L^2_r} \\
\label{E:Sob2} \| f\|_{L^6_r} \lesssim \omega^{1/6} \|Sf\|_{L^2_r} \\
\label{E:Sob3} \| \nabla_r f \|_{L_r^6} \lesssim \omega^{2/3} \|S^2 f\|_{L_r^2} \\
\label{E:Sob4} \| f \|_{L_r^\infty} \lesssim \omega^{1/4} \|S^2 f \|_{L_r^2}
\end{gather}
\end{lemma}
The factors of $\omega$ appearing here are seen to be optimal by taking $f(x,z) = g(x) h_\omega(z)$, where $g(x)$ is a smooth bump function.  Then $S^2f = (1-\Delta_x) g(x)h_{\omega}(z)$ and hence
$$\| Sf \|_{L^2_r}^2 = \la S^2f, f\ra = \la (1-\Delta_x)gh_\omega, h_\omega\ra = (\|g\|_{L_x^2}^2 + \|\nabla_x g\|_{L_x^2}^2)\|h_\omega\|_{L_z^2}^2$$
which is $\omega$-independent.  Also, $\|S^2 f\|_{L^2_r} = \| (1-\Delta_x)g \|_{L^2_x}$ is $\omega$-independent.  On the other hand, it is apparent that $\|\nabla_r f\|_{L_r^2} = \omega^{1/2}$, $\|f\|_{L_r^6} = \omega^{1/6}$, $\| \nabla_r f \|_{L_r^6} = \omega^{2/3}$ and $\|f\|_{L_r^\infty} = \omega^{1/4}$, which demonstrates sharpness of the  estimates.

\begin{proof}
Recall $I=P_0+P_1$.  First, we establish
\begin{gather}
\label{E:Sob1-2} \| \nabla_r P_1 f \|_{L_r^2} \lesssim  \|S f\|_{L^2_r} \\
\label{E:Sob2-2} \| P_1 f\|_{L^6_r} \lesssim  \|S f\|_{L^2_r} \\
\label{E:Sob3-2} \| \nabla_r P_1 f \|_{L_r^6} \lesssim  \|S^2 f\|_{L_r^2} \\
\label{E:Sob4-2} \| P_1 f \|_{L_r^\infty} \lesssim \|S^2 f \|_{L_r^2}
\end{gather}
Note that $P_j S^2 = S^2 P_j$.  By the definition of $S$,
$$ P_1 (1-\Delta_r+ \omega^2z^2)P_1 = S^2P_1 + \omega P_1$$ 
By spectral considerations $2\omega P_1 \leq P_1 S^2$, and hence
\begin{equation}
\label{E:Sob20}
\underbrace{P_1 (1-\Delta_r + \omega^2z^2) P_1}_{\text{all terms positive}} \lesssim S^2P_1 
\end{equation}
Since $[P_1 (1-\Delta_x)P_1, S^2 P_1]=0$ and $P_1(1-\Delta_x)P_1 \leq S^2P_1$ (from \eqref{E:Sob20}), we have by Lemma \ref{L:op-stuff}\eqref{I:op-3}
\begin{equation}
\label{E:Sob25}
P_1 (1-\Delta_x)^2 P_1 \lesssim S^4 P_1
\end{equation}
Since $[P_1 (-\partial_z^2 + \omega^2z^2)P_1, S^2P_1]=0$ and $P_1 (-\partial_z^2 + \omega^2z^2)P_1 \lesssim S^2P_1$  (from \eqref{E:Sob20}), we have by Lemma \ref{L:op-stuff}\eqref{I:op-3}
\begin{equation}
\label{E:Sob30}
P_1 (-\partial_z^2 + \omega^2z^2)^2P_1 \lesssim S^4P_1
\end{equation}
Expanding and ``integrating by parts''
\begin{equation}
\label{E:Sob31}
(-\partial_z^2 + \omega^2z^2)^2 = \underbrace{\partial_z^4 - 2\omega^2 \partial_z z^2 \partial_z + \omega^4z^4}_{\text{terms all positive}} + B + B^*
\end{equation}
where $B \defeq -2\omega^2 \partial_z z$.  We claim 
\begin{equation}
\label{E:Sob32}
P_1(B+B^*)P_1 \lesssim S^4 P_1
\end{equation}
Since $\| \partial_z P_1 f\|_{L_r^2} \lesssim \|S P_1 f \|_{L_r^2}$ and $\omega \|z P_1 f \|_{L_r^2} \lesssim \|SP_1 f\|_{L_r^2}$, it follows by Cauchy-Schwarz that
$$
 \omega^2 | \Re \la  \partial_z P_1 f, zP_1  f \ra | \lesssim  \omega \|S P_1 f\|_{L_r^2}^2 \lesssim \|S^2 P_1 f\|_{L_r^2}^2
$$
which is equivalent to \eqref{E:Sob32}.  By \eqref{E:Sob30}, \eqref{E:Sob31}, \eqref{E:Sob32}, we obtain
\begin{equation}
\label{E:Sob26}
P_1 (\partial_z^4) P_1 \lesssim S^4 P_1
\end{equation}
Now, \eqref{E:Sob25}, \eqref{E:Sob26} imply
\begin{equation}
\label{E:Sob21}
P_1 (1-\Delta_r)^2 P_1 \lesssim S^4 P_1
\end{equation}
Then \eqref{E:Sob1-2}, \eqref{E:Sob2-2}, \eqref{E:Sob3-2}, \eqref{E:Sob4-2} follow from Sobolev embedding and \eqref{E:Sob20}, \eqref{E:Sob21}.  For example, to prove \eqref{E:Sob3-2}, we apply 3D Sobolev embedding and \eqref{E:Sob21} to obtain
$$\| \nabla_r P_1 f \|_{L_r^6} \lesssim \| \Delta_r P_1 f \|_{L_r^2} \lesssim \|S^2 P_1 f\|_{L_r^2} \lesssim \| S^2 f\|_{L_r^2} \,.$$
Next we prove
\begin{gather}
\label{E:Sob1-0} \| \nabla_r P_0 f \|_{L_r^2} \lesssim \omega^{1/2} \|Sf\|_{L^2_r} \\
\label{E:Sob2-0} \| P_0f\|_{L^6_r} \lesssim \omega^{1/6} \|Sf\|_{L^2_r} \\
\label{E:Sob3-0} \| \nabla_r P_0f \|_{L_r^6} \lesssim \omega^{2/3} \|S^2 f\|_{L_r^2} \\
\label{E:Sob4-0} \| P_0f \|_{L_r^\infty} \lesssim \omega^{1/4} \|S^2 f \|_{L_r^2}
\end{gather}
Recall that
\begin{equation}
\label{E:Sob45}
P_0 f(x,z) = \int_{z'} f(x,z') h_\omega(z') \,dz' \; h_\omega(z) = \la f(x,\cdot), h_\omega \ra_{z'} \; h_\omega(z)
\end{equation}
We have 
\begin{equation}
\label{E:Sob40}
\nabla_x P_0 f(x,z) = \la \nabla_x f(x,\cdot), h_\omega \ra \; h_\omega(z)
\end{equation}
By Cauchy-Schwarz,
\begin{equation}
\label{E:Sob35}
\| \nabla_x P_0f \|_{L_r^2} \lesssim \| \nabla_x f \|_{L_r^2} \lesssim \|S f\|_{L_r^2}
\end{equation}
Also,
\begin{equation}
\label{E:Sob42}
\partial_z P_0 f(x,z) = \la f(x,\cdot) , h_\omega \ra \; \partial_z h_\omega(z)
\end{equation}
and hence by Cauchy-Schwarz,
\begin{equation}
\label{E:Sob36}
\| \partial_z P_0f\|_{L_r^2} \lesssim \omega^{1/2} \| f \|_{L_r^2}
\end{equation}
\eqref{E:Sob35} and \eqref{E:Sob36} together imply \eqref{E:Sob1-0}.
By Cauchy-Schwarz, Minkowski, and 2D Sobolev,
\begin{align*}
\| P_0 f\|_{L_r^6} &\leq \| \la f(x,\cdot) , h_\omega \ra_{z'} \|_{L_x^6} \|h_\omega \|_{L_z^6} \\
&\lesssim \| f\|_{L_x^6L_z^2} \|h_\omega\|_{L_z^2} \|h_\omega \|_{L_z^6} \\
&\lesssim \| (1-\Delta_x)^{1/2} f \|_{L_r^2} \; \omega^{1/6}
\end{align*}
Since $(1-\Delta_x) \leq S^2$, we obtain \eqref{E:Sob2-0} as a consequence of the previous estimate.  Next, we prove \eqref{E:Sob3-0}.
By \eqref{E:Sob40}, Cauchy-Schwarz, Minkowski, and 2D Sobolev,
\begin{align*}
\| \nabla_x P_0 f \|_{L_r^6} &\lesssim \| \la \nabla_x f(x,\cdot), h_\omega \ra \|_{L_x^6} \|h_\omega \|_{L_z^6} \\
& \lesssim \| \nabla_x f\|_{L_x^6L_z^2} \; \omega^{1/6} \\
& \lesssim \|(-\Delta_x)^{5/6} f \|_{L_r^2} \; \omega^{1/6}
\end{align*}
Since $(-\Delta_x)^{5/3} \leq (1-\Delta_x)^2 \leq S^4$, we obtain 
\begin{equation}
\label{E:Sob41}
\| \nabla_x P_0 f \|_{L_r^6} \lesssim \omega^{1/6}\| S^2 f\|_{L_r^2}
\end{equation}
By \eqref{E:Sob42}, Cauchy-Schwarz, Minkowski, and 2D Sobolev,
\begin{align*}
\| \partial_z P_0 f \|_{L_r^6} &\lesssim \| \la f(x,\cdot), h_\omega \ra \|_{L_x^6} \|\partial_z h_\omega \|_{L_z^6} \\
&\lesssim \| f\|_{L_x^6L_z^2} \;\omega^{2/3} \\
& \lesssim \| (1-\Delta_x)^{1/3} f\|_{L_r^2} \; \omega^{2/3}
\end{align*}
Since $(1-\Delta_x)^{2/3} \leq (1-\Delta_x)^2 \leq S^4$, we obtain
\begin{equation}
\label{E:Sob43}
\| \partial_z P_0 f\|_{L_r^6} \lesssim \| S^2 f\|_{L_r^2} \; \omega^{2/3}
\end{equation}
Combining \eqref{E:Sob41} and \eqref{E:Sob43}, we obtain \eqref{E:Sob3-0}.
Next, we prove \eqref{E:Sob4-0}.  By \eqref{E:Sob45} and 2D Sobolev,
\begin{align*}
\| P_0 f \|_{L_r^\infty} &\lesssim \| \la f(x,\cdot), h_\omega \ra \|_{L_x^\infty} \|h_\omega \|_{L_z^\infty} \\
&\lesssim \| f\|_{L_x^\infty L_z^2} \; \omega^{1/4} \\
& \lesssim \| (1-\Delta_x)^{\frac12+\epsilon} f \|_{L_r^2} \; \omega^{1/4}
\end{align*}
Since $(1-\Delta_x)^{1+2\epsilon} \leq (1-\Delta_x)^2 \leq S^4$, we obtain \eqref{E:Sob4-0} as a consequence of the previous estimate.

Note that combining \eqref{E:Sob1-2}--\eqref{E:Sob4-2} and \eqref{E:Sob1-0}--\eqref{E:Sob4-0} yeilds \eqref{E:Sob1}--\eqref{E:Sob4}.
\end{proof}

Let
$$\tilde S = (1-\Delta_x +\omega (-1-\partial_z^2 + z^2))^{1/2}$$

\begin{lemma}
\label{L:coercivity}
\begin{gather}
\label{E:tilde-S-1} \tilde S^2 \gtrsim 1-\Delta_r \\
\label{E:tilde-S-2} \tilde S^2 P_1 \geq P_1 ( 1-\Delta_x - \omega \partial_z^2 +\omega z^2) P_1 \\
\label{E:tilde-S-3} \tilde S^2 P_1  \geq  \omega P_1
\end{gather}
\end{lemma}
\begin{proof}
Directly from the definition of $\tilde S$, we have
\begin{equation}
\label{E:tS-6}
\underbrace{P_1 (1-\Delta_x - \omega \partial_z^2 + \omega z^2)P_1}_{\text{all terms positive}} \leq \omega P_1 + \tilde S^2P_1
\end{equation}
By spectral considerations
\begin{equation}
\label{E:tS-7}
2\omega P_1 \leq \omega(-1-\partial_z^2+z^2)P_1 \leq \tilde S^2 P_1
\end{equation}
Combining \eqref{E:tS-6} and \eqref{E:tS-7} yields \eqref{E:tilde-S-2}.  Also, \eqref{E:tilde-S-3} follows from \eqref{E:tS-7}.  Next, we establish \eqref{E:tilde-S-1} using \eqref{E:tilde-S-2}.  It is immediate that
\begin{equation}
\label{E:tS-1}
\tilde S^2 \geq (1-\Delta_x)
\end{equation}
On the other hand, since $P_0$ is just projection onto the smooth function $e^{-z^2}$,
\begin{equation}
\label{E:tS-2}
P_0(-\partial_z^2)P_0 \lesssim 1 \leq \tilde S^2
\end{equation}
By \eqref{E:tilde-S-2},
\begin{equation}
\label{E:tS-3}
P_1(-\partial_z^2)P_1 \leq \tilde S^2 P_1 \leq \tilde S^2
\end{equation}
By Lemma \ref{L:op-stuff}\eqref{I:op-1}, \eqref{E:tS-2}, \eqref{E:tS-3}, 
\begin{equation}
\label{E:tS-4}
-\partial_z^2 \lesssim \tilde S^2
\end{equation}
The claimed inequality \eqref{E:tilde-S-1} follows from  \eqref{E:tS-1} and \eqref{E:tS-4}.
\end{proof}

\begin{lemma}
\label{L:trace-of-tp-kernel}
Suppose $\sigma: L^2(\mathbb{R}^{3k}) \to L^2(\mathbb{R}^{3k})$ has kernel 
$$\sigma(\mathbf{r}_k, \mathbf{r}_k') = \int \psi(\mathbf{r}_k, \mathbf{r}_{N-k})\overline{\psi}(\mathbf{r}_k', \mathbf{r}_{N-k})\, d\mathbf{r}_{N-k} \,,$$
for some $\psi \in L^2(\mathbb{R}^{3N})$, and let $A,B:L^2(\mathbb{R}^{3k}) \to L^2(\mathbb{R}^{3k})$.  Then the composition $A\sigma B$  has kernel 
$$(A\sigma B)(\mathbf{r}_k, \mathbf{r}_k') = \int (A\psi)(\mathbf{r}_k, \mathbf{r}_{N-k}) (\overline{B^* \psi})(\mathbf{r}_k', \mathbf{r}_{N-k}) \, d\mathbf{r}_{N-k}$$   
It follows that 
$$\tr A\sigma B = \la A\psi, B^* \psi \ra \,.$$
\end{lemma}

Let $\mathcal{K}_k$ denote the class of compact operators on $L^2(\mathbb{R}^{3k})$, $\mathcal{L}^1_k$ denote the trace class operators on $L^2(\mathbb{R}^{3k})$, and $\mathcal{L}^2_k$ denote the Hilbert-Schmidt operators on $L^2(\mathbb{R}^{3k})$.  We have
$$\mathcal{L}_k^1 \subset \mathcal{L}_k^2 \subset \mathcal{K}_k$$
For an operator $J$ on $L^2(\mathbb{R}^{3k})$, let $|J| = (J^*J)^{1/2}$ and denote by $J(\mathbf{r}_k, \mathbf{r}'_k)$ the kernel of $J$ and $|J|(\mathbf{r}_k, \mathbf{r}'_k)$ the kernel of $|J|$, which satisfies $|J|(\mathbf{r}_k, \mathbf{r}_k') \geq 0$.  Let 
$$\mu_1 \geq \mu_2  \geq \cdots \geq 0$$ 
be the eigenvalues of $|J|$ repeated according to multiplicity (the \emph{singular values} of $J$).  Then
$$\| J \|_{\mathcal{K}_k} =  \| \mu_n \|_{\ell^\infty_n} = \mu_1 = \| \, |J| \,\|_{\op} = \|J\|_{\op}$$
$$\| J \|_{\mathcal{L}^2_k} =  \| \mu_n \|_{\ell^2_n} = \|J(\mathbf{r}_k, \mathbf{r}_k')\|_{L^2(\mathbf{r}_k, \mathbf{r}_k')} = (\tr J^*J)^{1/2}$$
$$\| J \|_{\mathcal{L}^1_k} = \| \mu_n \|_{\ell^1_n} = \|  |J|(\mathbf{r}_k,\mathbf{r}_k) \|_{L^1({\mathbf{r}_k})} = \tr |J|$$ 
The topology on $\mathcal{K}_k$ coincides with the operator topology, and $\mathcal{K}_k$ is a closed subspace of the space of bounded operators on $L^2(\mathbb{R}^{3k})$.  

\begin{lemma}
\label{L:compact-operator-truncation}
Let $\chi$ be a smooth function on $\mathbb{R}^3$ such that $\chi(\xi) =1$ for $|\xi|\leq 1$ and $\chi(\xi)=0$ for $|\xi| \geq 2$.  Let
$$(Q_M f)(\mathbf{r}_k) = \int e^{i\mathbf{r}_k \cdot \mathbf{\xi}_k} \prod_{j=1}^k \chi(M^{-1} \xi_j) \hat f(\mathbf{\xi}_k) \, d\mathbf{\xi}_k$$
With respect to the spectral decomposition of $L^2(\mathbb{R})$ corresponding to the operator $H_j = -\partial_{z_j}^2 + z_j^2$, let $Z^j_M$ be the orthogonal projection onto the sum of the first $M$ eigenspaces (in the $z_j$ variable only).  Let
$$R_M = \prod_{j=1}^k Z_M^j$$
\begin{enumerate}
\item Suppose that $J$ is a compact operator.  Then $J_M \defeq R_M Q_M J  Q_M R_M \to J$
in the operator norm.  
\item $H_jJ_M$, $J_MH_j$, $\Delta_{r_j} J_M$ and $J_M \Delta_{r_j}$ are all bounded.
\item There exists a countable dense subset $\{T_i\}$ of the closed unit ball in the space of bounded operators on $L^2(\mathbb{R}^{3k})$ such that each $T_i$ is compact and in fact for each $i$ there exists $M$ (depending on $i$) such that $T_i =  R_M Q_M T_i  Q_M R_M$.
\end{enumerate}
\end{lemma}
\begin{proof}
(1) If $S_n\to S$ strongly and $J\in \mathcal{K}_k$, then $S_nJ \to SJ$ in the operator norm and $JS_n \to JS$ in the operator norm.  (2) is straightforward.  For (3), start with a subset $\{Y_n\}$ of the closed unit ball in the space of bounded operators on $L^2(\mathbb{R}^{3k})$ such that each $Y_n$ is compact.  Then let $\{T_i\}$ be an enumeration of the set $R_M Q_M Y_n Q_M R_M$ where $M$ ranges over the dyadic integers.  By (1) this collection will still be dense.
\end{proof}

\section{Deducing Theorem \protect\ref{Theorem:3D->2D BEC
(Nonsmooth)} from Theorem \protect\ref{Theorem:3D->2D BEC}}
\label{A:equivalence}

The argument presented here which deduces Theorem \ref{Theorem:3D->2D BEC
(Nonsmooth)} from Theorem \ref{Theorem:3D->2D BEC} has been used in all the 
$n$D to $n$D work. We refer the readers to them for more details. We first
give the following proposition.

\begin{proposition}
\label{Prop:approximation of initial}Assume $\tilde{\psi}_{N,\omega }(0)$
satisfies (a), (b) and (c) in Theorem \ref{Theorem:3D->2D BEC (Nonsmooth)}.
Let $\chi \in C_{0}^{\infty }\left( \mathbb{R}\right) $ be a cut-off such
that $0\leqslant \chi \leqslant 1$, $\chi \left( s\right) =1$ for $
0\leqslant s\leqslant 1$ and $\chi \left( s\right) =0$ for $s\geqslant 2.$
For $\kappa >0,$ we define an approximation of $\tilde{\psi}_{N,\omega }(0)$
by 
\begin{equation*}
\tilde{\psi}_{N,\omega }^{\kappa }(0)=\frac{\chi \left( \kappa \left( \tilde{
H}_{N,\omega }-N\omega \right) /N\right) \tilde{\psi}_{N,\omega }(0)}{
\left\Vert \chi \left( \kappa \left( \tilde{H}_{N,\omega }-N\omega \right)
/N\right) \tilde{\psi}_{N,\omega }(0)\right\Vert }.
\end{equation*}
This approximation has the following properties:

(i) $\tilde{\psi}_{N,\omega }^{\kappa }(0)$ verifies the energy condition
\begin{equation*}
\langle \tilde{\psi}_{N,\omega }^{\kappa }(0),(\tilde{H}_{N,\omega }-N\omega
)^{k}\tilde{\psi}_{N,\omega }^{\kappa }(0)\rangle \leqslant \frac{2^{k}N^{k}
}{\kappa ^{k}}.
\end{equation*}

(ii)
\begin{equation*}
\sup_{N,\omega }\left\Vert \tilde{\psi}_{N,\omega }(0)-\tilde{\psi}
_{N,\omega }^{\kappa }(0)\right\Vert _{L^{2}}\leqslant C\kappa ^{\frac{1}{2}
}.
\end{equation*}

(iii) For small enough $\kappa >0$, $\tilde{\psi}_{N,\omega }^{\kappa }(0)$
is asymptotically factorized as well
\begin{equation*}
\lim_{N,\omega \rightarrow \infty }\limfunc{Tr}\left\vert \tilde{\gamma}
_{N,\omega }^{\kappa ,(1)}(0,x_{1},z_{1};x_{1}^{\prime },z_{1}^{\prime
})-\phi _{0}(x_{1})\overline{\phi _{0}}(x_{1}^{\prime
})h(z_{1})h(z_{1}^{\prime })\right\vert =0,
\end{equation*}
where $\tilde{\gamma}_{N,\omega }^{\kappa ,(1)}\left( 0\right) $ is the
marginal density associated with $\tilde{\psi}_{N,\omega }^{\kappa }(0),$
and $\phi _{0}$ is the same as in assumption (b) in Theorem \ref
{Theorem:3D->2D BEC (Nonsmooth)}.
\end{proposition}

\begin{proof}
Proposition \ref{Prop:approximation of initial} follows the same proof as 
\cite[Proposition 9.1]{E-S-Y5} if one replaces $H_{N}$ by $(\tilde{H}
_{N,\omega }-N\omega )$ and $\hat{H}_{N}$ by
\begin{equation*}
\sum_{j=2}^{N}(-\Delta _{x_{j}}+\omega (-1+-\partial
_{z_{j}}^{2}+z_{j}^{2}))+\frac{1}{N}\sum_{1<i<j\leq N}V_{N,\omega
}(r_{i}-r_{j}).
\end{equation*}
\end{proof}

Via (i) and (iii) of Proposition \ref{Theorem:3D->2D BEC}, $\tilde{\psi}
_{N,\omega }^{\kappa }(0)$ verifies the hypothesis of Theorem \ref
{Theorem:3D->2D BEC} for small enough $\kappa >0.$ Therefore, for $\tilde{
\gamma}_{N,\omega }^{\kappa ,(1)}\left( t\right) ,$ the marginal density
associated with $e^{it\tilde{H}_{N,\omega }}\tilde{\psi}_{N,\omega }^{\kappa
}(0),$ Theorem \ref{Theorem:3D->2D BEC} gives the convergence 
\begin{equation}
\lim_{\substack{ N,\omega \rightarrow \infty  \\ N\geqslant \omega ^{v(\beta
)+\varepsilon }}}\limfunc{Tr}\left\vert \tilde{\gamma}_{N,\omega }^{\kappa
,(k)}(t,\mathbf{x}_{k},\mathbf{z}_{k};\mathbf{x}_{k}^{\prime },\mathbf{z}
_{k}^{\prime })-\dprod\limits_{j=1}^{k}\phi (t,x_{j})\overline{\phi }
(t,x_{j}^{\prime })h_{1}(z_{j})h_{1}(z_{j}^{\prime })\right\vert =0.
\label{convergence:smooth}
\end{equation}
for all small enough $\kappa >0,$ all $k\geqslant 1$, and all $t\in \mathbb{R
}$. 

For $\tilde{\gamma}_{N,\omega }^{(k)}\left( t\right) $ in Theorem \ref
{Theorem:3D->2D BEC (Nonsmooth)}, we notice that, $\forall J^{(k)}\in 
\mathcal{K}_{k}$, $\forall t\in \mathbb{R}$, we have
\begin{eqnarray*}
&&\left\vert \limfunc{Tr}J^{(k)}\left( \tilde{\gamma}_{N,\omega
}^{(k)}\left( t\right) -\left\vert \phi \left( t\right) \otimes
h_{1}\right\rangle \left\langle \phi \left( t\right) \otimes
h_{1}\right\vert ^{\otimes k}\right) \right\vert  \\
&\leqslant &\left\vert \limfunc{Tr}J^{(k)}\left( \tilde{\gamma}_{N,\omega
}^{(k)}\left( t\right) -\tilde{\gamma}_{N,\omega }^{\kappa ,(k)}\left(
t\right) \right) \right\vert +\left\vert \limfunc{Tr}J^{(k)}\left( \tilde{
\gamma}_{N,\omega }^{\kappa ,(k)}\left( t\right) -\left\vert \phi \left(
t\right) \otimes h_{1}\right\rangle \left\langle \phi \left( t\right)
\otimes h_{1}\right\vert ^{\otimes k}\right) \right\vert  \\
&=& \text{I}+\text{II}.
\end{eqnarray*}
Convergence \eqref{convergence:smooth} then takes care of $\text{II}$. To handle $\text{I}$
, part (ii) of Proposition \ref{Theorem:3D->2D BEC} yields 
\begin{equation*}
\left\Vert e^{it\tilde{H}_{N,\omega }}\tilde{\psi}_{N,\omega }(0)-e^{it
\tilde{H}_{N,\omega }}\tilde{\psi}_{N,\omega }^{\kappa }(0)\right\Vert
_{L^{2}}=\left\Vert \tilde{\psi}_{N,\omega }(0)-\tilde{\psi}_{N,\omega
}^{\kappa }(0)\right\Vert _{L^{2}}\leqslant C\kappa ^{\frac{1}{2}}
\end{equation*}
which implies
\begin{equation*}
I=\left\vert \limfunc{Tr}J^{(k)}\left( \tilde{\gamma}_{N,\omega
}^{(k)}\left( t\right) -\tilde{\gamma}_{N,\omega }^{\kappa ,(k)}\left(
t\right) \right) \right\vert \leqslant C\left\Vert J^{(k)}\right\Vert
_{op}\kappa ^{\frac{1}{2}}.
\end{equation*}
Since $\kappa >0$ is arbitrary, we deduce that
\begin{equation*}
\lim_{\substack{ N,\omega \rightarrow \infty  \\ N\geqslant \omega ^{v(\beta
)+\varepsilon }}}\left\vert \limfunc{Tr}J^{(k)}\left( \tilde{\gamma}
_{N,\omega }^{(k)}\left( t\right) -\left\vert \phi \left( t\right) \otimes
h_{1}\right\rangle \left\langle \phi \left( t\right) \otimes
h_{1}\right\vert ^{\otimes k}\right) \right\vert =0.
\end{equation*}
i.e. as trace class operators 
\begin{equation*}
\tilde{\gamma}_{N,\omega }^{(k)}\left( t\right) \rightarrow \left\vert \phi
\left( t\right) \otimes h_{1}\right\rangle \left\langle \phi \left( t\right)
\otimes h_{1}\right\vert ^{\otimes k}\text{ weak*.}
\end{equation*}
Then again, the Gr\"{u}mm's convergence theorem upgrades the above weak*
convergence to strong. Thence, we have concluded Theorem \ref{Theorem:3D->2D
BEC (Nonsmooth)} via Theorem \ref{Theorem:3D->2D BEC} and Proposition \ref
{Prop:approximation of initial}.

\end{document}